\documentclass[reqno]{amsart}
\usepackage[foot]{amsaddr}
\makeatletter
\def\input@path{{configs/}}
\makeatother
\usepackage{preamble}
\usepackage{multirow}
\usepackage{pdflscape}
\usepackage{etoolbox}
\usepackage[dvipsnames]{xcolor}
\usepackage[normalem]{ulem}
\newcommand{\matt}{\mathrm{mat}}
\newcommand{\Rnd}{\R^{n\times \dots \times n}}
\newcommand{\Rnond}{\R^{n_1\times \dots \times n_d}}

\IfClassLoadedTF{cas-sc}{}{%
  \usepackage[bbgreekl]{mathbbol}
  \DeclareSymbolFontAlphabet{\mathbbm}{bbold}
}
\DeclareSymbolFontAlphabet{\mathbb}{AMSb}%

\newcommand{\mytensor}[1]{\mathbf{#1}}
\newcommand{\mytensorgreek}[1]{\boldsymbol{#1}}
\newcommand{\mytensormatrix}[1]{\mathsf{#1}}
\newcommand{\mytensormatrixgreek}[1]{\mathsf{#1}}

\newcommand{\tA}{\mytensor{A}}

\newcommand{\tB}{\mytensor{B}}

\newcommand{\tD}{\mytensor{D}}

\newcommand{\tG}{\mytensor{G}}

\newcommand{\tK}{\mytensor{K}}

\newcommand{\tM}{\mytensor{M}}

\newcommand{\tU}{\mytensor{U}}
\newcommand{\tV}{\mytensor{V}}
\newcommand{\tW}{\mytensor{W}}
\newcommand{\tX}{\mytensor{X}}
\newcommand{\tY}{\mytensor{Y}}
\newcommand{\tZ}{\mytensor{Z}}

\newcommand{\teta}{\mytensorgreek{\eta}}
\newcommand{\txi}{\mytensorgreek{\xi}}

\newcommand{\mtA}{\mytensormatrix{A}}
\newcommand{\mtB}{\mytensormatrix{B}}

\newcommand{\mtD}{\mytensormatrix{D}}

\newcommand{\mtG}{\mytensormatrix{G}}
\newcommand{\mtH}{\mytensormatrix{H}}
\newcommand{\mtI}{\mytensormatrix{I}}

\newcommand{\mtK}{\mytensormatrix{K}}
\newcommand{\mtL}{\mytensormatrix{L}}
\newcommand{\mtM}{\mytensormatrix{M}}

\newcommand{\mtP}{\mytensormatrix{P}}

\newcommand{\mtS}{\mytensormatrix{S}}

\newcommand{\mtV}{\mytensormatrix{V}}
\newcommand{\mtW}{\mytensormatrix{W}}

\newcommand{\mtLambda}{\mytensormatrixgreek{\Lambda}}

\usepackage{stmaryrd}

\newcommand{\idxrange}[1]{\llbracket #1 \rrbracket}

\newcommand{\ttrank}{\mathrm{rank}_\mathrm{TT}}
\newcommand{\ur}{\underline{r}}

\newcommand{\Gop}{\mathrm{G}}
\newcommand{\Hop}{\mathrm{H}}

\newcommand{\Projml}{\Proj^{\manif,L^2}}
\newcommand{\Projnl}{\Proj^{\cmanif,L^2}}
\newcommand{\Projsl}{\Proj^{\sphere,L^2}}
\newcommand{\Projm}{\Proj^{\manif}}
\newcommand{\Projvecm}{\Proj^{\vecmanif}}
\newcommand{\Projn}{\Proj^{\cmanif}}

\usepackage{ wasysym }

\newcommand{\cmanif}{\mathcal{C}^{\ur}}
\newcommand{\func}{E}
\newcommand{\X}{X}
\newcommand{\manif}{\mathcal{M}^{\ur}}
\newcommand{\vecmanif}{\mathfrak{M}^{\ur}}
\newcommand{\veccmanif}{\mathfrak{C}^{\ur}}
\newcommand{\bbN}{\mathbb{N}}

\newcommand{\sphere}{\mathcal{S}}

\newcommand{\analysis}{\mathbb{A}^{\Phi}}
\newcommand{\synthesis}{\mathbb{S}^{\Phi}}

\newcommand{\is}{i_1, \dots, i_d}

\newcommand{\oned}{\mathrm{1d}}

\newcommand{\Phioned}{\Phi^{\oned}}
\newcommand{\tphioned}{\widetilde{\phi}^{\oned}}
\newcommand{\Voned}{\mathrm{V}^{\oned}}

\newcommand{\matau}{\mtA_{\tU}}
\newcommand{\mass}{\mtM}

\newcommand{\stiff}{\mtS}
\newcommand{\stiffoned}{\stiff^{\oned}}

\newcommand{\matauh}{\hat{\mtA}_{\tU}}
\newcommand{\stiffh}{\hat{\stiff}}
\newcommand{\funch}{\hat{\func}}
\newcommand{\massh}{\hat{\mass}}

\newcommand{\masshtensrec}{\hat{\tM}_{\mathrm{rec}}}
\newcommand{\masshtenssqrtrec}{\hat{\tM}_{\mathrm{rec}}^{1/2}}
\newcommand{\masshoned}{\massh^{\oned}}
\newcommand{\stiffonedlag}{\stiff^{\oned}_{\mathrm{L}}}
\newcommand{\vmrec}{\vm_{\mathrm{rec}}}

\newcommand{\auh}[2]{\hat{a}_u(#1, #2)}
\newcommand{\innercont}[2]{(#1, #2)}
\newcommand{\innerquad}[2]{(#1, #2)_h}

\newcommand{\precon}{\mtB}
\renewcommand{\P}{\mathcal{P}}
\newcommand{\PAPMr}{\P_{\matau}}
\newcommand{\PAhPMr}{\P_{\matauh}}
\newcommand{\Pprecon}{\P_{\precon}}
\newcommand{\Ppreconinv}{\P_{\precon^{-1}}}
\newcommand{\preconapp}{\tilde{\precon}^{-1}}
\newcommand{\Ppreconapp}{\P_{\preconapp}}
\newcommand{\tKinv}{\tK_{\mathrm{rec}}}
\newcommand{\tKinvapp}{\widetilde{\tK}_{\mathrm{rec}}}

\newcommand{\Hspace}{\mathbf{H}}
\newcommand{\Lspace}{\mathbf{L}}
\newcommand{\Dp}{\mathcal{D}(p)}
\newcommand{\dinner}[2]{\langle\!\langle #1, #2 \rangle\!\rangle}
\newcommand{\OBshort}{\mathcal{OB}}
\newcommand{\OB}{\OBshort_{\mN}(p, \Hspace)}
\newcommand{\OBN}{\OBshort_{\mN}(p, \X_N)}
\newcommand{\Projob}{\Proj^{\OBshort}}
\newcommand{\Projobl}{\Proj^{\OBshort,L^2}}
\newcommand{\Projng}[1]{\Proj^{\cmanif,#1}}
\newcommand{\vecOBshort}{\mathfrak{OB}}
\newcommand{\vecOB}{\vecOBshort_{\mN}(p)}
\newcommand{\Projvecob}{\Proj^{\vecOBshort}}
\newcommand{\Projvecn}{\Proj^{\veccmanif}}
\newcommand{\gradFN}{\grad[\mathrm{F},\veccmanif]}
\newcommand{\PAhPMrphi}{\P_{\matauhphi}}
\newcommand{\PAhPN}{\P^{\veccmanif}_{\matauhphi}}
\newcommand{\auhphi}[2]{\hat{a}_{\tPhi}(#1, #2)}
\newcommand{\matauhphi}{\hat{\mtA}_{\tPhi}}
\newcommand{\gradF}{\grad[\mathrm{F}]}
\newcommand{\Aop}{\mathrm{A}}

\newcommand{\tPhi}{\mytensorgreek{\Phi}}

\newcommand{\stiffhp}{\hat{\stiff}_{p}}
\newcommand{\masshp}{\hat{\mass}_{p}}
\newcommand{\couplp}{\mK_{p}}
\newcommand{\masshtensrecp}{\hat{\tM}_{\mathrm{rec},p}}
\newcommand{\hphi}{\widehat{\phi}}

\let\phi\varphi
\newcommand{\inta}{0}
\newcommand{\intb}{1}
\newcommand{\baseint}{(\inta,\intb)}

\usepackage{placeins}

\usepackage[numbers,sort]{natbib}

\definecolor{cm}{rgb}{0.4,0.4,0.1}
\definecolor{ib}{rgb}{1,0.5,0} 
\makeatletter
\renewcommand\paragraph{\@startsection{paragraph}{4}%
  \z@\z@{-\fontdimen2\font}{\normalfont\itshape}}
\makeatother

\hypersetup{
    colorlinks=true,
    allcolors=green!50!black,
    urlcolor=red!50!black
}

\newcommand{\ttround}{\mathrm{round}}
\newcommand{\transp}[2]{\mathcal{T}_{#2\leftarrow#1}}

\newcommand{\theabstract}{%
  This work is concerned with the numerical computation of Gross--Pitaevskii
  ground states using the tensor train (TT) format. We regard the problem as the
  minimization of the Gross--Pitaevskii energy functional under a unit mass
  constraint, and propose a first-order Riemannian optimization scheme that
  preserves both the mass constraint and the low-rank representation throughout
  the optimization, by operating on the manifold of unit-norm functions of fixed
  TT-rank. On this manifold, we derive the energy-adaptive Riemannian gradient
  and show how to compute it efficiently in the TT format. Spatial
  discretization is performed using a spectral method with numerical
  integration, which allows the required computations to be carried out
  efficiently while preserving the low-rank format. The method is shown to
  substantially reduce computational time compared to full-rank computations,
  while maintaining accuracy for both single- and multicomponent
  Gross--Pitaevskii equations.%
}

\begin{document}

\title[Riemannian optimization on low-rank TT manifolds for the GPE]{Riemannian
optimization on low-rank tensor train manifolds for the Gross--Pitaevskii
equation}

\author[I.~Bioli]{Ivan Bioli$^{\dagger,\ddagger}$}
\email{ivan.bioli@unipv.it}

\author[C.~Marcati]{Carlo Marcati$^{\S}$}
\email{carlo.marcati@univ-lyon1.fr}

\author[M.~Rakhuba]{Maxim Rakhuba$^{\mathparagraph}$}
\email{mrakuba@hse.ru}

\address{$^{\dagger}$Dipartimento di Matematica, Università di Pavia, 27100
Pavia, Italy}
\address{$^{\ddagger}$Dipartimento di Ingegneria Civile e Architettura,
Università di Pavia, 27100 Pavia, Italy}
\address{$^{\S}$Université Lyon 1, Centrale Lyon, INSA Lyon, Université Jean
Monnet, CNRS, ICJ UMR5208, 69622 Villeurbanne, France}
\address{$^{\mathparagraph}$HSE University, Faculty of Computer Science, 109028 Moscow, Russian Federation}

\thanks{The work of IB has received funding from the European Union's Horizon
Europe research and innovation programme under the Marie Sk\l{}odowska-Curie
Action MSCA-DN-101119556 (IN-DEEP). IB is a member of the Gruppo Nazionale
Calcolo Scientifico -- Istituto Nazionale di Alta Matematica (GNCS-INdAM).
CM acknowledges the support of the French
National Research Agency (ANR) through grant ANR-24-CPJ1-0099-01
and of the Italian Ministry of University and Research (MUR)
through the PRIN 2022 PNRR project P2022NC97R, funded by the
European Union -- Next Generation EU
and the PRIN 2022 project 202292JW3F}

\begin{abstract}
    \theabstract

    \bigskip\noindent
    \textsc{Keywords.} Gross--Pitaevskii equation, tensor train, low-rank,
    Riemannian optimization, energy-adaptive gradient.
\end{abstract}

\maketitle

\sloppy

\section{Introduction}
\label{sec:intro}
We consider the minimization problem
\begin{equation} \label{eq:min}
    \min \bigg\{ \func(v): v\in \X \text{ and } \|v\|_{L^2(\Omega)} = 1\bigg\},
\end{equation}
in the space $\X\coloneqq H_0^1(\Omega)$ and with $\Omega$ a product domain in
dimension $d\in \{2, 3\}$. The energy functional is given by
\begin{equation} \label{eq:energy}
    \func (v) \coloneqq
    \frac 12 \int_{\Omega} |\nabla v|^2 +
    \frac 12 \int_{\Omega} V v^2 +
    \frac{\beta}{2} \int_{\Omega} v^4
\end{equation}
for some value $\beta>0$ and a potential $V:\Omega\to \R$ that is bounded from below.
A minimizer of \eqref{eq:min} is a stationary state of the
Gross--Pitaevskii (GP) equation (without rotation) and a ground state of
certain Bose--Einstein condensates under the Hartree--Fock
approximation.

A popular and effective way to minimize the Gross--Pitaevskii energy \eqref{eq:energy} is to use
methods based on gradient flows or Riemannian optimization.
The goal is to find a minimizer by following a curve of
decreasing energy on the unit sphere in $L^2(\Omega)$, so that the mass
constraint is always satisfied. Gradient flows and Riemannian optimization
often yield schemes that are equivalent at the discrete level. The descent
direction is then determined by the specific metric chosen on the manifold or,
equivalently, by the space in which the gradient is calculated.
Different choices have been proposed and studied
in the literature (see \cite{henningGrossPitaevskiiEquation2025} for a review
and the references in \Cref{sec:previous-work}).

In this paper, we speed up Riemannian optimization methods for the computation
of ground states of \eqref{eq:min}, by using tensor compression. We represent
the (discretized) quantum states and operators in Tensor Train (TT) format. We
thus reduce the cost of a discretization that has $n$ degrees of freedom per
coordinate direction from $n^d$ to a quantity that is linear in $n$ and $d$.
To do this, we start from the observation that TTs of fixed rank are a
manifold \cite{holtz_manifolds_2012} and that the intersection of the unit
sphere with TTs of fixed rank is itself a manifold (see
\Cref{prop:intersection}). It follows that Riemannian optimization can be
carried out on the intersection manifold, i.e., on functions of unit norm that can
be represented as TTs of low-rank.

The manuscript is then devoted to the derivation and implementation of
Riemannian line-search optimization on this intersection manifold. We focus on
the fixed-rank version of the ``energy-adapted''
\cite{altmann_energy-adaptive_2022} optimization, where the Riemannian metric
is given by the bilinear form associated with the linearization of the GP
Hamiltonian. We define and show how to efficiently compute all the objects that
are necessary for the implementation of the scheme. Since the computation of
the descent direction at each iteration requires the solution of a linear
system, we also show how to precondition this system, both in the full-size
problem and in the tensor-compressed one, where some empirical modifications
are required in order to retain the efficiency of the preconditioner. We then
investigate the multicomponent GP equation, whose efficient low-rank
minimization requires additional care compared to the single-component case.

Finally, we compare the computation of the ground state with low-rank tensors
with the full-size counterpart on a number of test cases, which differ in the
choice of the potential $V$ and of the constant $\beta$, for both the single- and
multicomponent equations. We observe substantial reductions in computation
time in most of the numerical tests, with the time needed to reach an
approximation of the ground state often reduced by an order of magnitude.
\subsection{Previous work}
\label{sec:previous-work}
We bring together techniques from two separate fields: on one
side, the minimization of GP energy functionals with Sobolev gradient flows; on
the other, dynamical low-rank tensor representations and Riemannian optimization
on manifolds of fixed-rank tensors.
Tensor trains have been recently used for GP equation, especially in their
Quantized/Quantics version (QTT), see \cite{Niedermeier2026,Connor2026}, though
not in a Riemannian framework, with the imaginary time evolution of
\cite{BouComas2025} being the closest to our algorithms.
The QTT approximation
of solutions to nonlinear eigenproblems has been investigated in \cite{MRS2022},
with a focus on singular potentials.
For a broader perspective on low-rank tensor methods for partial differential equations, we refer to~\cite{Bachmayr2023Low}.
\paragraph{Gradient flows and Riemannian optimization for Gross-Pitaevskii}
Gradient flows for the minimization of Schrödinger
energy functionals have been used in
\cite{GarciaRipollPerezGarcia2001,bao2004computing}, the latter being the
foundational work for Gross--Pitaevskii. Different gradients have subsequently
been considered in
\cite{DanailaKazemi2010}, where the
authors have introduced a specific Sobolev inner product
to deal with rotation. A point of view closer to Riemannian optimization is
developed in \cite{danaila_computation_2017,antoine_efficient_2017}, where the
nonlinear conjugate gradient is also used.
The energy-adaptive metric that we use here has been introduced and analyzed in
\cite{henning_sobolev_2020}. The methods based on Sobolev
gradient flows/Riemannian optimization with different metrics have been shown to
converge in \cite{chenConvergenceSobolevGradient2024,chenFullyDiscretizedSobolev2024}.
We mention also \cite{heid_gradient_2021}, where adaptive spatial
discretizations are combined with the minimization process, and
\cite{AltmannJmethod}, where the J-method is applied to the GP equation. Generalizations of
the methods considered in this paper to
Kohn--Sham \cite{altmann_energy-adaptive_2022} and multicomponent equations
\cite{altmann_riemannian_2025} have been developed; the latter is particularly
relevant for our work on the same subject. We conclude by mentioning the recent
interest in second-order minimization schemes
\cite{AltmannPeterseimStykel2024,AiHenningYadavYuan2026,altmann_riemannian_2025}.
Extending the methods introduced in this paper to second-order techniques is an
interesting direction for future research.
\paragraph{Riemannian optimization on manifolds of tensors of fixed rank}
As already mentioned, tensors of prescribed TT rank are a smooth manifold, and
Riemannian optimization algorithms can be implemented directly in the TT format,
including nonlinear conjugate-gradient methods \cite{steinlechner_riemannian_2016}
and exploit automatic differentiation \cite{novikov_automatic_2022}, and algorithms for
computing several eigenstates of high-dimensional
Hamiltonians~\cite{rakhuba2019low}.
Automatic differentiation can be used to compute the required Riemannian derivatives, as in the present work.
Closely related ideas arise in dynamical low-rank approximation,
where the ambient evolution is projected onto the tangent space of a low-rank
manifold \cite{KochLubich2007}; extensions to TT representations have been developed in \cite{lubichTimeIntegrationTensor2015,HaegemanEtAl2011,Haegeman2016Unifying}.

\subsection{Contributions}
The main contributions of this paper can be summarized as follows:
\begin{itemize}
    \item We introduce Riemannian optimization schemes
          for the GP energy over fixed-rank tensor trains.
    \item We derive efficient ways to compute the Riemannian gradient for the
          single- and multicomponent problems.
    \item We propose and test different preconditioning strategies, and adapt them
          to the low-rank case.
    \item We test the method numerically and show that it allows for substantial
          gains in computational time on the tests we consider.
\end{itemize}
\subsection{Structure of the paper}
The rest of the paper is organized as follows. \Cref{sec:formulation} introduces
the notation and the continuous minimization problem. \Cref{sec:findim}
describes the finite-dimensional setting, the TT format, and the Riemannian
manifolds involved. \Cref{sec:riemannian} presents the Riemannian optimization
algorithms and the computation of the Riemannian gradient on the manifold of
unit-norm fixed TT-rank functions. \Cref{sec:numerical_discretization} and
\Cref{sec:precond} detail the discretization and the preconditioners.
\Cref{sec:multicomponent} extends the whole construction to the multicomponent
case, and \Cref{sec:numexp} presents the numerical experiments.

\section{Problem formulation}
\label{sec:formulation}

\subsection{Notation}
We denote the positive integers $\bbN = \{1, 2, \dots\}$ and $\bbN_0 = \{0\}
    \cup \bbN$. For all $n\in \bbN$, we write $\idxrange{1,n} = \{1, \dots,
    n\}$. For $p\in [1, \infty]$, $d\in\bbN$, and $\Omega\subset\R^d$,
$L^p(\Omega)$ denotes the space of $p$-summable functions. $H^k(\Omega)$
denotes the usual Sobolev space of order $k\in\bbN$.

We denote vectors $\va\in\R^n$ by bold lowercase letters,
and tensors (with ``tensor'' meaning a multi-dimensional array)
$\tA\in\R^{n_1\times\dots\times n_d}$ by bold capitals, with
entries $\tA(i_1,\dots,i_d)$ or $\tA_{i_1,\dots,i_d}$. The Frobenius inner
product of $\tX,\tY$ is written $\inner{\tX}{\tY}$, and the Hadamard
(entry-wise) product as $\tX \odot \tY$. The vectorization operator
$\vecc:\R^{n_1\times\dots\times n_d}\to\R^{n_1\cdots n_d}$ stacks tensor
entries in reverse lexicographic order, with inverse $\matt:\R^{n_1\cdots
        n_d}\to\R^{n_1\times\dots\times n_d}$. Linear operators are denoted by
sans-serif letters $\mtL:\R^{n_1\times\dots\times n_d}\to
    \R^{n_1\times\dots\times n_d}$, and, with a slight abuse of notation, we
often identify them with their matrix representation in $\R^{(n_1\cdots
        n_d)\times(n_1\cdots n_d)}$.
Finally, for a tensor $\tA$, $\diag(\tA)$
denotes the operator corresponding to the Hadamard product with $\tA$, i.e.,
$\diag(\tA): \R^{n_1\times \dots\times n_d }\to \R^{n_1\times\dots\times n_d}$
such that $\diag(\tA)\tB = \tA \odot \tB$.

\subsection{The continuous minimization problem}
We consider, throughout the paper, the minimization problem \eqref{eq:min} under
the hypotheses that $V\in L^p(\Omega)$ for $p>d/2$ and $V\geq 0$ a.e.~in $\Omega$.
The extension to $V\geq V_{\min{}}$ with
$V_{\min{}}\in \R$ is straightforward, as replacing $V$ with $V - V_{\min{}}$
only causes a shift in the energy, without affecting the minimizer.

We define, for all $u \in X$, the bilinear form
\begin{equation}
    \label{eq:bilinear}
    a_u(v, w) \coloneqq
    \int_{\Omega} \nabla v \cdot \nabla w \,dx +
    \int_{\Omega} V v w \,dx +
    2\beta\int_{\Omega} u^2 v w \,dx,
    \qquad \forall v, w \in X.
\end{equation}
It satisfies
\begin{equation}
    \label{eq:au-Eprime}
    a_u(u, v) = E'(u)[v], \qquad \forall v \in X,
\end{equation}
with $E'(u) \in X'$ denoting the Fréchet derivative of $E$ at $u$. For fixed $u
    \in X$, the bilinear form $a_u : X \times X \to \R$ is
continuous and coercive. In particular, there exist constants $C_1, C_2 > 0$
such that, for all $u, v, w \in X$,
\begin{equation}
    \label{eq:au_coercivity}
    \begin{split}
        a_u(v, v) & \geq C_1 \norm{\nabla v}_{L^2(\Omega)}^2,                    \\
        a_u(v, w) & \leq C_2 \bigl(1 + \norm{\nabla u}_{L^{2}(\Omega)}^{2}\bigr)
        \norm{\nabla v}_{L^2(\Omega)} \norm{\nabla w}_{L^2(\Omega)}.
    \end{split}
\end{equation}
The energy \eqref{eq:energy} has two minimizers $u$ and $-u$, with $u\geq 0$
a.e.~in $\Omega$ \cite{Cances2010}. In the following, we will refer to $u$ as
the minimizer of the energy. The minimizers also satisfy the elliptic nonlinear
eigenvector equation
\begin{equation*}
    -\Delta u + Vu + 2\beta u^3 = \lambda u
\end{equation*}
for some value $\lambda\in \R$, which is the Euler--Lagrange equation associated
with the original minimization problem.

\section{Finite-dimensional discretization and manifolds}
\label{sec:findim}
In this section we introduce the finite-dimensional discretization of the
problem, we recall the Tensor Train (TT) format, and we present the Riemannian
manifolds involved in the optimization problem.

\subsection{Tensor product discretization}
Let $X_N$ be a
finite-dimensional subspace of $X$ with $\dim(X_N) = N$. The associated
discrete problem reads
\begin{equation}
    \label{eq:min-disc}
    \min \bigg\{ \func(v_N) : v_N \in \X_N, \ \|v_N\|_{L^2(\Omega)} = 1 \bigg\}.
\end{equation}
Assume that $X_N$ is a tensor product space. For simplicity, we restrict the
presentation to tensor products of identical spaces, although the derivations
extend to the case of different spaces in each dimension. Let $N = n^d$ and
\[
    X_N = \bigotimes_{k=1}^d \Voned, \qquad \Voned \subseteq H^1_0(\baseint), \,\,\dim(\Voned) = n.
\]
Let $\Phioned = \{\phi_1, \dots, \phi_n\}$ be a basis of $\Voned$. The tensor
product basis of $X_N$ is then
\[
    \Phi = \Big\{ \phi_{i_1, \dots, i_d} \coloneqq \phi_{i_1} \otimes \dots \otimes \phi_{i_d} \Big\}_{i_1, \dots, i_d \in \idxrange{1, n}},
\]
where with the $\otimes$ notation we mean $\phi_{i_1, \dots, i_d}(\vx) = \phi_{i_1}(x_1) \cdots \phi_{i_d}(x_d)$ for $\vx = (x_1, \dots, x_d) \in \Omega$.
The specific choice of basis that we make in this paper will be detailed in
\Cref{sec:numerical_discretization}.
Each $u \in X_N$ can be identified with the $d$-dimensional tensor $\tU \in
    \Rnd$ of its coefficients in the basis $\Phi$. This relation is expressed
through the analysis and synthesis operators defined below. The synthesis
operator $\synthesis_N \colon \mathbb{R}^{n\times \dots \times n} \to \X_N$
is given by
\[
    \synthesis_N(\tU) = \sum_{i_1,\dots,i_d = 1}^n \tU(i_1, \dots, i_d) \, \phi_{i_1, \dots, i_d}.
\]
Its inverse is the analysis operator
\[
    \analysis_N \coloneqq \left(\synthesis_N\right)^{-1} \colon \X_N \to \mathbb{R}^{n\times \dots\times n}.
\]
We refer to $\tU = \analysis_N(u)$ as the \emph{tensor representation of} $u$,
and to $u = \synthesis_N(\tU)$ as the \emph{function represented by} $\tU$.

In the following, for ease of notation, we sometimes omit the dependence on $N$
in the subscripts.

\subsection{The manifold of unit-norm fixed TT-rank functions}
\label{sec:manifolds}
In this section we describe the manifold of unit-norm fixed TT-rank functions,
obtained as the intersection of the manifold of fixed TT-rank functions and the $L^2$-sphere manifold.

\paragraph{The manifold of fixed TT-rank functions.}
Let $r_j \in \bbN$ for $j=0, \dots, d$.
We say that $\tA\in\Rnond$ can be represented as a
TT-tensor~\cite{oseledets2011tensor} of TT-rank $\ur = (r_0, \dots,r_{d})$, with
$r_0=r_d =1$, if there exist $G_k\colon \{1,\dots,n_k\}\to \R^{r_{k-1}\times
        r_k}$, $k=1,\dots, d$, such that
\begin{equation}
    \label{eq:TT-generic}
    \tA(i_1, \dots, i_d) = G_1(i_1)\dots G_d(i_d),\qquad \forall (i_1,
    \dots, i_d)\in \idxrange{1, n_1}\times\dots\times\idxrange{1, n_d}.
\end{equation}
The maps $G_k$ are called TT-cores and can equivalently be seen as
three-dimensional arrays. The TT-rank $\ttrank(\tA)$ of a tensor $\tA\in\Rnond$
is the smallest $\ur$ such that $\tA$ can be represented as a TT-tensor of rank
$\ur$. The TT-rank of a tensor is always well defined
\cite{holtz_manifolds_2012,oseledets2011tensor}. Throughout, we write $n =
    \max_k n_k$ and $r = \max_k r_k$. The representation \eqref{eq:TT-generic} is
described by $\calO(d n r^2)$ parameters instead of the $n^d$ entries of the
full tensor \cite{oseledets2011tensor}: the exponential dependence on the
dimension is replaced by a linear one, which makes it possible to handle
high-dimensional tensors efficiently, provided their (approximate) TT-rank is
small.

Following \cite[Section~4.1]{steinlechner_riemannian_2016}, it is convenient to
split $\tA$ into a left and a right part through the \emph{interface matrices}
\begin{equation}
    \label{eq:interface}
    \begin{aligned}
        \tA_{\leq k}(i_1, \dots, i_k) & = G_1(i_1)G_2(i_2)\cdots G_k(i_k),
                                      & \tA_{\leq k}                                                & \in \R^{n_1 n_2\cdots n_k \times r_k},         \\
        \tA_{\geq k}(i_k, \dots, i_d) & = \bigl[G_k(i_k)G_{k+1}(i_{k+1})\cdots G_d(i_d)\bigr]^\top,
                                      & \tA_{\geq k}                                                & \in \R^{n_k n_{k+1}\cdots n_d \times r_{k-1}},
    \end{aligned}
\end{equation}
where we slightly abuse the notation, as $\tA_{\leq k}(i_1, \dots, i_k)$ is the
row of $\tA_{\leq k}$ indexed by $(i_1, \dots, i_k)$, hence a row vector of
length $r_k$, and analogously $\tA_{\geq k}(i_k, \dots, i_d)$ is a row vector of
length $r_{k-1}$. We set $\tA_{\leq 0} = \tA_{\geq d+1} = 1$, so that
\eqref{eq:TT-generic} reads
\begin{equation}
    \label{eq:TT-interface}
    \tA(i_1, \dots, i_d) = \tA_{\leq k-1}(i_1, \dots, i_{k-1})\, G_k(i_k)\,
    \tA_{\geq k+1}(i_{k+1}, \dots, i_d)^\top, \qquad k = 1, \dots, d.
\end{equation}
We further denote by $\mtG_k^{\mathsf{L}} \in \R^{r_{k-1}n_k\times r_k}$ and
$\mtG_k^{\mathsf{R}} \in \R^{r_{k-1}\times n_k r_k}$ the \emph{left} and
\emph{right unfoldings} of the core $G_k$, obtained by merging its first two and
its last two indices, respectively. The tensor $\tA$ is called
\emph{$\mu$-orthogonal} if
\begin{equation}
    \label{eq:left-orth}
    \begin{aligned}
        (\mtG_k^{\mathsf{L}})^\top \mtG_k^{\mathsf{L}} = \mtI_{r_k},
         & \quad \text{and hence} \quad \tA_{\leq k}^\top \tA_{\leq k} = \mtI_{r_k},
         &                                                                               & \text{for all } k = 1, \dots, \mu-1, \\
        \mtG_k^{\mathsf{R}} (\mtG_k^{\mathsf{R}})^\top = \mtI_{r_{k-1}},
         & \quad \text{and hence} \quad \tA_{\geq k}^\top \tA_{\geq k} = \mtI_{r_{k-1}},
         &                                                                               & \text{for all } k = \mu+1, \dots, d,
    \end{aligned}
\end{equation}
and it is called \emph{left-orthogonal} if $\mu = d$ and \emph{right-orthogonal}
if $\mu = 1$. Such a representation is obtained from an arbitrary one with
$\calO(dnr^3)$ operations, by a sequence of QR factorizations of the unfoldings
of the cores \cite[Algorithm~4.1]{steinlechner_riemannian_2016}.

Tensors of fixed TT-rank form a manifold
\cite{holtz_manifolds_2012,steinlechner_riemannian_2016}, which we denote
\begin{equation*}
    \vecmanif = \big\{\tA \in \Rnond: \ttrank(\tA)= \ur \big\}.
\end{equation*}
Given a tensor $\tA \in \vecmanif$ with cores $G_1, \dots, G_d$ as in
\eqref{eq:TT-generic}, the tangent space to $\vecmanif$ at $\tA$ is given by
\begin{multline}
    \label{eq:delta-parametrization}
    \Tang_{\tA} \vecmanif=  \big\{ \tB:\tB(i_1, \dots, i_d) = \delta G_1(i_1) G_2(i_2)\cdots
    G_d(i_d) + \ldots + G_1(i_1)\cdots G_{d-1}(i_{d-1})\delta G_d(i_d)\\
    \text{ with }\delta G_j (i_j)\in \R^{r_{j-1}\times r_j}
    \text{ for all }j=1, \dots, d
    \big\}.
\end{multline}
A unique
representation of the elements of the tangent space can be obtained by
imposing the gauge conditions
\begin{equation}
    \label{eq:gauge}
    \left[
        \delta G_j, G_j \right] = 0_{r_j\times r_j} \qquad \text{for all }j=1, \dots, d-1,
\end{equation}
where for two cores $V$ and $W$, the product $[V, W]_{mn} =
    \sum_{i}\sum_{k}V(i)_{km}W(i)_{kn}$ multiplies the columns of the left
unfoldings of the cores. A tangent vector is thus itself a TT tensor, of TT rank
at most $2\ur$, and is stored through the $\delta G_k$ with the same
$\calO(dnr^2)$ complexity as a point of $\vecmanif$
\cite[Section~4.4]{steinlechner_riemannian_2016}. We say that the tensor is
parametrized by the $\delta$-cores $\delta G_1, \dots, \delta
    G_d$, implicitly assuming the gauge condition \eqref{eq:gauge}.

We introduce the manifold of TT-rank
$\ur = (r_0, \dots, r_d)$ functions in $\X_N$ as the set of functions whose
tensor representation has fixed TT-rank $\ur$:
\[
    \manif = \left\{ u \in \X_N : \ttrank\bigl(\analysis_N u\bigr) = \ur \right\}.
\]
Note that this definition depends only on the underlying one-dimensional spaces,
not on the specific choice of bases, since a change of basis is represented by a
TT-matrix of rank one.
Denote by $\vecmanif$ the set of TT-tensors of fixed TT-rank $\ur$ in $\Rnd$.
Since $\manif = \synthesis_N(\vecmanif)$ and the synthesis operator is a linear
isomorphism, $\manif$ is an embedded submanifold of $\X_N$, with tangent space
at $u \in \manif$ given by
\[
    \Tang_u \manif = \synthesis_N \bigl( \Tang_{\analysis_N u} \vecmanif \bigr).
\]

\paragraph{The \texorpdfstring{$L^2$}{L2}-sphere.}
We introduce the spherical manifold associated with the unit $L^2$ norm
constraint:
\[
    \sphere = \left\{ u \in \X_N : \|u\|_{L^2(\Omega)} = 1 \right\},
\]
whose tangent space at $u \in \sphere$ is
\[
    \Tang_u \sphere = \left\{ v \in \X_N : (v,u)_{L^2 (\Omega)} = 0 \right\}.
\]

\paragraph{The intersection manifold.}
We define the set of unit-norm and fixed TT-rank functions as the intersection
\begin{equation}
    \label{eq:intersection-manifold}
    \cmanif = \manif \cap \sphere.
\end{equation}
Crucially for what follows, this is itself a manifold.
\begin{proposition}
    \label{prop:intersection}
    $\cmanif$ is an embedded submanifold of $\X_N$. For all $u \in \cmanif$,
    \[
        \Tang_u (\manif \cap \sphere) =
        \Tang_u \manif \cap \Tang_u \sphere.
    \]
\end{proposition}
\begin{proof}
    The manifolds $\manif$ and $\sphere$ intersect transversally.
    Let $u\in\manif \cap \sphere$.
    Then $\mathrm{dim}(\Tang_u\sphere) = N-1$, with
    \[
        \Tang_u \sphere + \{\alpha u, \alpha\in\mathbb{R}\} = \X_N.
    \]
    For any $\alpha \in\R$, $\alpha \analysis_N u \in \Tang_{\analysis_N u} \vecmanif$ (choose $\delta G_k = 0$ for $k=1,\dots,d-1$ and $\delta G_d = \alpha G_d$).
    Therefore $\alpha u\in \Tang_u \manif$ and we obtain that
    \[
        \Tang_u \sphere + \Tang_u \manif = \X_N,
    \]
    which implies transversality and hence concludes the proof,
    by~\cite[Corollary 3.5.13]{Abraham1988}.
\end{proof}
We can therefore minimize the energy \eqref{eq:energy} on the intersection
manifold $\cmanif$ using Riemannian optimization. The description of the
algorithm is the subject of the next section.

\section{Riemannian optimization on the manifold of unit-norm fixed TT-rank functions}
\label{sec:riemannian}

We now describe how the minimization problem \eqref{eq:min-disc} is solved on
$\cmanif$. We first present, in \Cref{sec:linesearch}, the two line-search
methods we use, which are stated for a general Riemannian manifold.
\Cref{sec:tools} specifies the implementation of the methods on $\cmanif$, and \Cref{sec:rgrad}, which contains the bulk of the
derivation, is devoted to the Riemannian gradient.

\subsection{Riemannian line-search methods}
\label{sec:linesearch}

Given an initial guess $u^0 \in \cmanif$, a line-search Riemannian optimization
method on $\cmanif$ \cite{absil_optimization_2008,boumal_introduction_2023}
produces the iterates
\begin{equation}
    \label{eq:scheme}
    u^{k+1} = \Retr_{u^k}\bigl(\alpha_k \xi_k\bigr), \qquad k = 0, 1, \dots,
\end{equation}
where $\xi_k \in \Tang_{u^k}\cmanif$ is a search direction, $\alpha_k > 0$ is a
step size, and $\Retr$ is a retraction onto $\cmanif$. Building such an
iteration requires three ingredients: the Riemannian gradient $\grad\func(u)$,
which defines the search direction; the retraction $\Retr$; and, for the
conjugate gradient variant below, a transporter $\transp{u}{v}$ mapping tangent
vectors from $\Tang_u\cmanif$ to $\Tang_v\cmanif$. The last two are given in
\Cref{sec:tools}, the first in \Cref{sec:rgrad}.

\paragraph{Riemannian Gradient Descent (R-GD)}
The simplest choice for the descent direction is the negative Riemannian gradient,
\begin{equation}
    \label{eq:rgd}
    \xi_k = -\grad \func(u^k).
\end{equation}
The Riemannian gradient depends on the metric of the manifold. This manuscript
focuses on the metric induced by the $a_u$ bilinear form, see \eqref{eq:metric}
below. This gives the discrete counterpart of the energy-adaptive Sobolev
gradient flow of \cite{henning_sobolev_2020}, with the step size selected by
line search rather than fixed.

\paragraph{Riemannian Nonlinear Conjugate Gradient (R-NLCG)}
R-NLCG accelerates \eqref{eq:rgd} by adding to the negative gradient the
previous search direction, transported to the current tangent space:
\begin{equation} \label{eq:rnlcg}
    \xi_k = -g_k + \beta_k\, \transp{u^{k-1}}{u^k}(\xi_{k-1}),
    \qquad g_k \coloneqq \grad\func(u^k).
\end{equation}
We use the Hestenes--Stiefel rule with a non-negativity safeguard
\cite{sato_riemannian_2021}, namely
\begin{equation}
    \label{eq:betaHS}
    \beta_k = \max\left(0, \,
    \frac{\innersmall[u^k]{g_k}{y_k}}
    {\innersmall[u^k]{y_k}{\transp{u^{k-1}}{u^k}(\xi_{k-1})}}\right),
    \qquad
    y_k = g_k - \transp{u^{k-1}}{u^k}(g_{k-1}),
\end{equation}
where all inner products are taken in the $a_u$ metric \eqref{eq:bilinear} at the
current iterate $u^k$. If the direction produced by \eqref{eq:rnlcg} fails to
satisfy $\innersmall[u^k]{g_k}{\xi_k} < 0$, it is discarded and replaced by
$\xi_k = -g_k$, so that R-NLCG reduces to one R-GD step. The choice $\beta_k
    \equiv 0$ recovers R-GD, and the two methods therefore differ only in the search
direction: they solve the same flow, R-NLCG being an accelerated version of
R-GD.

\paragraph{Step size.}
In both methods $\alpha_k$ is chosen by Armijo backtracking adapted to
Riemannian optimization \cite[Definition~4.2.2]{absil_optimization_2008}, with
the initial trial step inferred from the previous iteration.

\subsection{Geometric tools on \texorpdfstring{$\cmanif$}{Nr}}
\label{sec:tools}
Following \cite{henning_sobolev_2020}, we define a Riemannian metric on the
tangent bundle of the manifold of unit-norm and fixed TT-rank functions, given
by
\begin{equation}
    \label{eq:metric}
    \forall u \in \cmanif \quad \left<v, w\right>_u = a_u(v, w),
    \qquad \forall v, w \in \Tang_u \cmanif.
\end{equation}
The fact that this is a Riemannian metric follows from the coercivity and
continuity of the bilinear form. The remaining two ingredients of
\eqref{eq:scheme} are the following. We discuss this choice of metric and
explicit formulas to compute the gradient in \Cref{sec:rgrad}.

\paragraph{$L^2$-orthogonal projection and transporter.}
To map tangent vectors between tangent spaces in \eqref{eq:rnlcg} and
\eqref{eq:betaHS} we use the $L^2$-orthogonal projection onto the target tangent
space,
\begin{equation}
    \label{eq:transporter}
    \transp{u}{v} = \restr{\Projnl_v}{\Tang_u \cmanif} \colon
    \Tang_u \cmanif \to \Tang_v \cmanif,
    \qquad u, v \in \cmanif.
\end{equation}
The transporter \eqref{eq:transporter} is cheap because $\Projnl_u$ factors
through the standard TT tangent-space projection, as follows.

\begin{lemma}
    \label{lem:proj-commute-single}
    For all $u \in \cmanif$,
    \[
        \Projnl_u = \Projsl_u \Projml_u = \Projml_u \Projsl_u,
    \]
    where $\Projml_u$ and $\Projsl_u$ denote the $L^2$-orthogonal projections onto
    $\Tang_u \manif$ and $\Tang_u \sphere$, respectively.
\end{lemma}
\begin{proof}
    Since $\norm{u}_{L^2(\Omega)} = 1$, the $L^2$-orthogonal projection onto
    $\Tang_u\sphere = \{v \in \X_N : (v,u)_{L^2(\Omega)} = 0\}$ is $\Projsl_u v = v - (u,v)_{L^2(\Omega)}\,u$.
    By the same arguments as in the proof of \Cref{prop:intersection}, $u \in \Tang_u\manif$, so that
    $\Projml_u u = u$ and, by self-adjointness of $\Projml_u$,
    \[
        (u, \Projml_u v)_{L^2(\Omega)} = (\Projml_u u, v)_{L^2(\Omega)} = (u, v)_{L^2(\Omega)}, \qquad \forall v \in \X_N .
    \]
    Both compositions therefore equal $v \mapsto \Projml_u v - (u,v)_{L^2}\,u$, and
    $\Projml_u$ and $\Projsl_u$ commute. Two commuting orthogonal projections
    compose to the orthogonal projection onto the intersection of their ranges,
    which by \Cref{prop:intersection} is $\Tang_u\manif \cap \Tang_u\sphere =
        \Tang_u\cmanif$.
\end{proof}
It follows that the computational cost of applying $\Projnl_u$ is the same as one
application of $\Projml_u$ followed by a
rank-one correction. The projection $\Projml_u$ of a TT tensor of rank
$\calO(\tilde{r})$ onto $\Tang_u\manif$ requires $\calO(dn\tilde{r}r^2)$
operations \cite[Section~4.4]{steinlechner_riemannian_2016}, and the correction
is a single inner product between TT tensors, at cost $\calO(dnr^2)$; the transporter
\eqref{eq:transporter} therefore costs $\calO(dnr^3)$ operations.

\paragraph{Retraction.}
For $u \in \cmanif$ and $\xi \in \Tang_u\cmanif$ we use the retraction obtained
by truncating back to TT rank $\ur$ and then normalizing,
\begin{equation}
    \label{eq:retraction}
    \Retr_u(\xi) = \frac{w}{\norm{w}_{L^2(\Omega)}},
    \qquad
    w = \synthesis_N\Bigl( \ttround_{\ur}\bigl( \analysis_N (u + \xi) \bigr) \Bigr),
\end{equation}
where $\ttround_{\ur}$ denotes TT rounding to rank $\ur$. Since $u + \xi$ has TT
rank at most $2\ur$, the rounding is performed by the TT-SVD algorithm at a cost
of $\calO(dnr^3)$ operations \cite[Section~3]{oseledets2011tensor}. The
retraction is therefore not more expensive than a single tangent-space
projection.

\subsection{The \texorpdfstring{$a_u$}{au} Riemannian gradient}
\label{sec:rgrad}

The Riemannian gradient of $\func$ at $u \in
    \cmanif$ with respect to the metric \eqref{eq:metric} is the unique element $\grad \func(u) \in
    \Tang_u \cmanif$ such that
\begin{equation}
    \label{eq:Riemannian-gradient}
    \left< \grad E(u), v \right>_u =
    E'(u)[v], \qquad \forall v \in \Tang_u \cmanif,
\end{equation}
where $E'(u)$ is the Fréchet derivative of $\func$ at $u$. Using
\eqref{eq:au-Eprime} and \eqref{eq:metric}, we obtain
\[
    a_u(\grad E(u), v) = a_u(u, v), \qquad \forall v \in \Tang_u \cmanif. 
\]
We introduce the $a_u$-orthogonal
projection onto $\Tang_u \cmanif$, that is, the operator $\Projn_u \colon \X_N \to \Tang_u \cmanif$ such that
\begin{equation*}
    a_u(\Projn_u v - v, w) = 0\qquad \text{for all }w\in\Tang_u\cmanif, \, v\in X_N.
\end{equation*}
It follows from the above equations that
\[
    \grad E(u) = \Projn_u u
\]
is the Riemannian gradient on the manifold $\cmanif$ equipped with the
$a_u(\cdot, \cdot)$ metric. A different choice of metric gives rise to a
different Riemannian gradient. The other metric that is relevant for this
manuscript is the one induced by the $H^1$ scalar product, which induces the
$H^1$ Sobolev gradient flow that we use in some of the numerical experiments. We
do not treat this method here theoretically, but we refer to
\cite{henning_sobolev_2020,chenConvergenceSobolevGradient2024,henningGrossPitaevskiiEquation2025} for a
presentation and analysis of the different gradient flows.

\paragraph{Computation of $\Projn_u$.}
In this section, we
derive a computable expression for the projection $\Projn_u$. The following
result, giving an explicit formula for the $a_u$-orthogonal projection onto
$\Tang_u\sphere$, is shown in~\cite{henning_sobolev_2020}.

\begin{lemma}
    Let $\Gop_u\colon X' \to X_N$ be such that $a_u(\Gop_u z, w) = (z, w)$ for
    all $z\in X'$, $w\in X_N$. Then, for all $u\in \sphere$, the operator
    $\Proj_u^\sphere\colon X_N \to \Tang_u \sphere$ defined by
    \[
        \Proj_u^\sphere v = v - \frac{a_u(\Gop_u u, v)}{a_u (\Gop_u u, \Gop_u u)} \Gop_u u =
        v - \frac{(u,v)}{(u,\Gop_u u)} \Gop_u u, \qquad \forall v\in X_N,
    \]
    is the $a_u$-orthogonal projection onto $\Tang_u \sphere$.
\end{lemma}

Here, we aim to provide an analogue for the projection onto $\Tang_u\cmanif$. We
start by introducing the Galerkin inverse of $A_u$ onto $\Tang_u\manif$.
\begin{definition}
    \label{def:Hop}
    For $u \in \manif$, we define $\Hop_u \colon X' \to \Tang_u \manif$ by
    \begin{equation}
        \label{eq:Hop}
        a_u (\Hop_u v, w) = (v, w), \qquad \forall w \in \Tang_u \manif.
    \end{equation}
\end{definition}
This operator is well defined thanks to the coercivity and continuity of $a_u$
on $\Tang_u\manif\subset X_N$ and the Lax--Milgram lemma. The next result gives
the desired expression for $\Projn_u$. We denote by $\Projm_u \colon X_N \to
    \Tang_u \manif$ the $a_u$-orthogonal projection onto $\Tang_u \manif$.
\begin{proposition}
    \label{prop:projN}
    For all $u\in \cmanif$, let $\Proj_u^{\cmanif}$ be defined by
    \begin{equation}
        \label{eq:proj_N_v}
        \Projn_u v = \Proj_u^{\manif} v - \frac{(u, \Proj_u^{\manif} v)}{(u,\Hop_u u)} \Hop_u u  = \left(\Id - \frac{\Hop_uu} { (u, \Hop_uu)}\Proj^{L^2}_u  \right) \circ \Proj_u^{\manif} v, \qquad \forall v\in X_N.
    \end{equation}
    Then $\Projn_u$ is the $a_u$-orthogonal projection onto $\Tang_u\cmanif$.
\end{proposition}

\begin{proof}
    \textit{Step 1.} For all $v\in \Tang_u\cmanif = \Tang_u\manif \cap \Tang_u \sphere$ we have
    \begin{equation*}
        \Proj_u^{\manif} v = v\qquad \text{and}\qquad(u, v)=0.
    \end{equation*}
    It follows that $\Projn_u$ defined as in \eqref{eq:proj_N_v} satisfies $\Projn_u v = v$ for all $v\in \Tang_u \cmanif$.
    \\
    \textit{Step 2.} We show that the image of $\Projn_u$ is contained in $\Tang_u\cmanif$.
    For all $v\in\X_N$, it holds that $\Projn_u v\in \Tang_u \sphere$, since
    \[
        (\Projn_u v, u) = (\Proj_u^{\manif}v, u) - \frac{(u, \Projm_u v)}{(u, \Hop_u u)} (u, \Hop_u u)
        = (\Projm_u v, u) - (u, \Projm_u v) = 0.
    \]
    Furthermore, since by definition $\Projm_u v \in \Tang_u\manif$ and $\Hop_u u\in \Tang_u\manif$, we
    also have $\Projn_u v \in \Tang_u \manif$.
    \\
    \textit{Step 3.} We show that $a_u(\Projn_u v - v, w) = 0$ for all $v\in\X_N$ and $w\in \Tang_u \cmanif$:
    \[
        a_u(v-\Projn_u v, w) = \underbrace{a_u(v - \Projm_u v, w)}_{=0 \text{ since } w\in \Tang_u\manif} -
        \frac{(u, \Projm_u v)}{(u, \Hop_u u)} a_u (\Hop_u u, w)
        = -\frac{(u, \Projm_u v)}{(u, \Hop_u u)} (u,w) = 0,
    \]
    where the last equality follows from $w\in \Tang_u \sphere$. This concludes the proof.
\end{proof}
\paragraph{Computation of $\Proj_u^{\manif}$ and $\Hop_u$.}

We now derive computable expressions for $\Proj_u^{\manif}$ and $\Hop_u$, which
form the building blocks for $\Projn_u$.

\begin{lemma}
    \label{lem:H-formula}
    Let $\calV$ be a finite-dimensional real inner product space, let $\Tang
        \subseteq \calV$ be a subspace, and let $\mtP\colon \calV \to \calV$ be
    the orthogonal projector onto $\Tang$, i.e.,
    \[
        \range(\mtP) = \Tang, \qquad \mtP^2 = \mtP = \mtP^*,
    \]
    where $\mtP^*$ is the adjoint of $\mtP$ with respect to the inner product of $\calV$.
    Let $\mtA \colon \calV \to \calV$ be a self-adjoint positive definite linear
    operator and let $\mtB \colon \calV \to \calV$ be an arbitrary linear
    operator. The unique linear operator $\mtH \colon \calV \to \calV$
    satisfying
    \begin{enumerate}[(i)]
        \item $\range(\mtH) \subseteq \Tang$,
        \item $\inner{\mtA \mtH \vz}{\vw} = \inner{\mtB \vz}{\vw}$ for all $\vw
                  \in \Tang$, $\vz \in \calV$,
    \end{enumerate}
    is
    \begin{equation}
        \label{eq:H-matrix-formula}
        \mtH =  \mtP (\mtP\mtA\mtP)^\dagger \mtP\mtB = (\mtP\mtA\mtP)^\dagger \mtP\mtB,
    \end{equation}
    where $(\cdot)^\dagger$ denotes the Moore--Penrose pseudo-inverse.
\end{lemma}
\begin{proof}
    The existence and uniqueness of $\mtH$ follow from the Riesz representation
    theorem since $\mtA$ is positive definite. Indeed, for fixed $\vz$, $\mtH\vz
        \in \Tang$ is the unique vector representing the functional $\vw \mapsto
        \inner{\mtB \vz}{\vw}$ on $\Tang$ in the $\mtA$-inner product.

    For the explicit formula, we check that $\mtH = \mtP(\mtP\mtA\mtP)^\dagger
        \mtP\mtB$ satisfies (i) and (ii). Note first that $\range(\mtH)
        \subseteq \range(\mtP) = \Tang$. For (ii), if $\vw \in \Tang$, then $\vw
        = \mtP\vw$, and, using that $\mtP$ is self-adjoint,
    \begin{align*}
        \inner{\mtA \mtH \vz}{\vw} - \inner{\mtB \vz}{\vw}
         & = \inner{(\mtA \mtH - \mtB) \vz}{\mtP\vw}
        = \inner{\bigl(\mtP\mtA\mtP (\mtP\mtA\mtP)^\dagger \mtP\mtB - \mtP\mtB\bigr) \vz}{\vw}    \\
         & = \inner{\bigl(\mtP\mtA\mtP (\mtP\mtA\mtP)^\dagger - \Id\bigr) \mtP\mtB \vz}{\vw} = 0,
    \end{align*}
    since $\mtP\mtA\mtP (\mtP\mtA\mtP)^\dagger$ is the orthogonal projection
    onto $\range(\mtP\mtA\mtP) = \range(\mtP) = \Tang$ and $\mtP\mtB\vz \in
        \Tang$.
\end{proof}

We now introduce the discretization of the $a_u$ and $L^2$ inner products on
$X_N$ with respect to the basis $\Phi$, yielding the operators $\matau$ and
$\mass$. These are defined as $\matau, \mass \colon \Rnd \to \Rnd$ such that for
all $u,w,z \in X_N$, with tensor representations $\tU, \tW, \tZ$, we have
\begin{equation*}
    a_u(w, z) = \inner{\matau \tW}{\tZ},
    \qquad \innerh[L^2]{w}{z} = \inner{\mass \tW}{\tZ}.
\end{equation*}

\begin{proposition}
    \label{prop:projau_Hu}
    Let $u \in \manif$, $z \in X_N$, and denote $\tU = \analysis_N u \in
        \vecmanif$, $\tZ = \analysis_N z \in \Rnd$. Let $\Projvecm_{\tU} \colon
        \Rnd \to \Tang_{\tU} \vecmanif$ be the orthogonal projection onto
    $\Tang_{\tU} \vecmanif$ with respect to the Frobenius inner product.
    Then
    \begin{align}
        \Projm_u z & = \synthesis_N \txi,
                   & \txi                  & = \Projvecm_{\tU} (\Projvecm_{\tU} \matau \Projvecm_{\tU})^\dagger \Projvecm_{\tU} \matau \tZ, \\
        \Hop_u z   & = \synthesis_N \teta,
                   & \teta                 & = \Projvecm_{\tU} (\Projvecm_{\tU} \matau \Projvecm_{\tU})^\dagger \Projvecm_{\tU} \mass \tZ.
    \end{align}
\end{proposition}
\begin{proof}
    Let $\Projm_u z = \synthesis_N \txi$ and $\Hop_u z = \synthesis_N \teta$,
    with $\txi,\teta \in \Tang_{\tU} \vecmanif$. For all $\tW \in \Tang_{\tU}
        \vecmanif$, and $w = \synthesis_N \tW \in \Tang_u\manif$, we have
    \begin{align*}
        a_u(\Projm_u z, w) & = \inner{\matau \txi}{\tW} = \inner{\matau \tZ}{\tW} = a_u(z, w),    \\
        a_u(\Hop_u z, w)   & = \inner{\matau \teta}{\tW} = \inner{\mass \tZ}{\tW} = (z, w)_{L^2}.
    \end{align*}
    Applying \Cref{lem:H-formula} with $\calV = \Rnd$ endowed with the Frobenius
    inner product, $\Tang = \Tang_{\tU}\vecmanif$, $\mtP = \Projvecm_{\tU}$,
    $\mtA = \matau$, and $\mtB = \matau$ and $\mtB = \mass$ respectively, yields
    the result.
\end{proof}

\Cref{prop:projau_Hu} is what makes the whole construction affordable. The
pseudo-inverse of $\Projvecm_{\tU}\matau\Projvecm_{\tU}$ is never formed: the
operator is symmetric positive definite on $\Tang_{\tU}\vecmanif$, a space of
dimension $\calO(dnr^2)$, and the two systems appearing in \Cref{prop:projau_Hu}
are solved by conjugate gradient. Each CG iteration consists of one application
of $\matau$ to a tangent vector, of TT rank at most $2\ur$, followed by one
projection $\Projvecm_{\tU}$, so its cost is the $\calO(dnr^3)$ of the
projection plus the cost of applying $\matau$, quantified in
\Cref{sec:numerical_discretization}. The overall cost of one Riemannian gradient
is this quantity times the number of CG iterations, which is why the
preconditioners of \Cref{sec:precond} are essential.

\section{Discretization by spectral elements}
\label{sec:numerical_discretization}
Riemannian optimization requires the evaluation of the energy $\func(u)$ and the
Riemannian gradient $\grad\func(u) = \Projn_u u$ for $u\in\cmanif$. By
\Cref{prop:projau_Hu,prop:projN}, computing $\Projn_u u$ amounts to applying
$\matau$ and $\mass$ to tensors in TT format and solving linear systems on
$\Tang_{\tU}\vecmanif$ with the operator $\PAPMr \coloneqq
    \Projvecm_{\tU}\,\matau \,\Projvecm_{\tU}$. Since $\PAPMr$ is positive definite
on $\Tang_{\tU}\vecmanif$, the natural solver is the (preconditioned) conjugate
gradient method. Because tangent vectors in $\Tang_{\tU}\vecmanif$ can be
represented as TT tensors of rank at most $2\ur$, solving systems with $\PAPMr$
again reduces to applying $\matau$ and the preconditioner to TT tensors. This
section details the discretization and efficient implementation of these
operations.

For the discretization, we follow
\cite{chenFullyDiscretizedSobolev2024,liu_simple_2024} and employ the Spectral
Element Method (SEM) with numerical integration (which can also be seen as a
generalized collocation method) on quadrilateral grids. Let $\Xi$ be a partition
of $\baseint$ into $E$ elements, and define $\Voned \subseteq H^1_0(\baseint)$
as the space of continuous piecewise polynomials of degree $k$, i.e.,
\[
    \Voned = \Big\{v_h \in C^0(\baseint)\,\,:\,\, v_h(\inta) = v_h(\intb) =0,\, \restr{v_h}{e} \in \mathbb{P}^k(e)\, \forall e \in \Xi\Big\}.
\]
In one dimension, as it is classical with spectral collocation/spectral element methods, we use as a basis for $\Voned$ the
Lagrangian interpolant basis functions associated with the $(k+1)$
Gauss--Lobatto--Legendre (GLL) points on each element, and we use the same points
for quadrature
\cite{CHQZ2,deville_high-order_2002,karniadakis_spectralhp_2005,maday_spectral_1989}.
Tensorization yields the multidimensional scheme.
Using $(k+1)$ GLL nodes per element gives $n = kE-1$ interior nodes
$\{x_i\}_{i=1}^n$. Let $\phi_i$ and $m_i$ be the basis function and quadrature
weight at $x_i$. In the following, we denote by $\innercont{\cdot}{\cdot}$ the exact $L^2$ inner
product and by $\innerquad{\cdot}{\cdot}$ its quadrature approximation.

Choosing the interpolation nodes to coincide with the quadrature points is
computationally advantageous, albeit at the price of introducing a quadrature
error. See \cite{CHQZ2} for a thorough treatment of spectral element methods with numerical integration. In the one-dimensional case, it makes the mass
matrix with quadrature diagonal,
\[
    \innerquad{\phi_i}{\phi_j} = m_i \delta_{ij},
\]
i.e., the mass matrix with quadrature equals $\diag(\vm)$ with $\vm = [m_1, \cdots,
    m_n]^\top$. We exploit this by rescaling the basis functions, taking as a basis
of $\Voned$
\begin{equation}
    \label{eq:rescaled-basis}
    \Phioned = \{\tphioned_1, \dots, \tphioned_n\}, \qquad
    \tphioned_i \coloneqq \frac{\phi_i}{\sqrt{m_i}},
\end{equation}
which is orthonormal with respect to the quadrature inner product, since
$\innerquad{\tphioned_i}{\tphioned_j} = m_i \delta_{ij} / \sqrt{m_i m_j} =
    \delta_{ij}$. Consequently, the mass matrix with quadrature is the identity
in every dimension,
\[
    \masshoned = \mtI, \qquad \massh = \mtI,
\]
so that the discrete $L^2$ inner product is simply the Frobenius inner product
of the coefficient tensors, $\innerquad{w}{z} = \inner{\tW}{\tZ}$. Note that in
this basis the coefficient tensor of $u$ carries the square root of the
quadrature weights,
\begin{equation}
    \label{eq:coeff-scaling}
    \tU(\is) = \sqrt{m_{i_1} \cdots m_{i_d}}\; u(x_{i_1}, \dots, x_{i_d}).
\end{equation}

The stiffness matrix $\stiffoned$ for the one-dimensional problem is exact,
since the $(k+1)$-point Gauss--Lobatto quadrature is of order $2k-1$:
\[
    [\stiffoned]_{i,j} = \innerquad{\nabla\tphioned_i}{\nabla\tphioned_j} = \innercont{\nabla\tphioned_i}{\nabla\tphioned_j},
\]
and it is obtained from the stiffness matrix $\stiffonedlag$ in the Lagrangian
basis by the symmetric scaling $\stiffoned = \diag(\vm)^{-1/2}\, \stiffonedlag
    \, \diag(\vm)^{-1/2}$.
The stiffness matrix in $d$ dimensions has the TT representation
\begin{equation*}
    \begin{split}
        \stiffh & = \stiffoned \otimes \mtI \otimes \cdots \otimes \mtI + \dots + \mtI \otimes  \cdots \otimes \mtI \otimes \stiffoned \\
                & =
        \begin{bmatrix}
            \stiffoned & \mtI
        \end{bmatrix}
        \Join
        \begin{bmatrix}
            \mtI       & 0    \\
            \stiffoned & \mtI
        \end{bmatrix}
        \Join
        \cdots
        \Join
        \begin{bmatrix}
            \mtI       & 0    \\
            \stiffoned & \mtI
        \end{bmatrix}
        \Join
        \begin{bmatrix}
            \mtI \\
            \stiffoned
        \end{bmatrix},
    \end{split}
\end{equation*}
which has TT rank $2$. Since applying a TT matrix of TT rank $R$ to a TT tensor
of rank $r$ yields a tensor of rank at most $Rr$ and costs $\calO(d n^2 R^2
    r^2)$ operations \cite[Section~4.3]{oseledets2011tensor}, applying $\stiffh$
therefore costs $\calO(dn^2r^2)$ operations and at most doubles the rank, whereas the same
product in the full format costs $\calO(n^{d+1})$ operations.

This choice also enables efficient integration of the nonlinear-in-$u$ terms in
$\func$ and $a_u$. While the mass matrix no longer appears, by
\eqref{eq:coeff-scaling} the quadrature weights still enter these terms, through
the rank-$1$ TT tensor of reciprocal weights
\begin{equation}
    \label{eq:Mrec}
    \masshtensrec = \otimes_{i=1}^d \vmrec, \qquad \vmrec = [1/m_1, \cdots, 1/m_n]^\top,
\end{equation}
whose application is a Hadamard product. Let
\[
    \tV(i_1,\dots,i_d) = V(x_{i_1},\dots,x_{i_d}), \qquad \tV\in \Rnd,
\]
be the tensor of potential values at quadrature nodes. For $u=\synthesis_N \tU$,
the discrete energy reads
\begin{equation}
    \label{eq:E_discrete}
    \begin{split}
        \funch(u) & \coloneqq
        \frac{1}{2} \innerquad{\nabla u}{\nabla u}
        + \frac{1}{2} \innerquad{Vu}{u}
        + \frac{\beta}{2} \innerquad{u^2}{u^2}                                             \\
                  & =
        \frac{1}{2} \innersmall{\tU}{\stiffh \tU}
        + \frac{1}{2} \innersmall{\tV \odot \tU}{\tU}
        + \frac{\beta}{2} \innersmall{\tU \odot \tU}{\masshtensrec \odot (\tU \odot \tU)}. \\
    \end{split}
\end{equation}
The discrete $a_u$ inner product is
\begin{equation}
    \label{eq:au_discrete}
    \begin{split}
        \auh{w}{z} & =
        \innerquad{\nabla w}{\nabla z}
        + \innerquad{Vw}{z}
        + 2\beta\innerquad{uw}{uz} \\
                   & =
        \inner{\tW}{\stiffh \tZ}
        + \innersmall{\tV \odot \tW}{\tZ}
        + 2 \beta \innersmall{\tU \odot \tW}{\masshtensrec \odot (\tU \odot \tZ)} \eqqcolon \innersmall{\tW}{\matauh \tZ}.
    \end{split}
\end{equation}
Explicitly,
\begin{equation}
    \label{eq:mat_au_discrete}
    \matauh = \stiffh + \diag(\tV) + 2\beta \diag(\masshtensrec \odot \tU \odot \tU).
\end{equation}
All operations in \eqref{eq:E_discrete} and \eqref{eq:au_discrete} can be
performed efficiently with TT tensors.

From now on, all inner products are replaced by their discrete counterparts. For
instance, the manifold $\cmanif$ is discretized as
\[
    \veccmanif = \{\tU \in \vecmanif : \inner{\tU}{\tU} = 1\},
\]
and in \Cref{prop:projN,prop:projau_Hu} the operators $\matau$ and $\mass$ are
replaced by $\matauh$ and $\massh = \mtI$. In particular, the mass drops out of
\Cref{prop:projau_Hu}, whose second identity becomes $\teta = \Projvecm_{\tU}
    (\Projvecm_{\tU} \matauh \Projvecm_{\tU})^\dagger \Projvecm_{\tU} \tZ$.

\paragraph{Computational complexity.} The rank behaviour is what dictates the
cost of computing the discrete energy \eqref{eq:E_discrete} and the discrete
$a_u$ inner product \eqref{eq:au_discrete}. A Hadamard product of two TT tensors
of ranks $r_1$ and $r_2$ has rank at most $r_1 r_2$ and is computed in $\calO(d
    n r_1^2 r_2^2)$ operations, while the Frobenius inner product of two TT tensors
of ranks $r_1$ and $r_2$ costs $\calO(d n r_1 r_2 \min(r_1,r_2))$
\cite[Section~4]{oseledets2011tensor}. Assume that $\tV$ has TT rank $R$, small
and independent of $n$ and $d$, as is the case for all the potentials of
\Cref{sec:numexp}. Then the linear terms of \eqref{eq:E_discrete} and
\eqref{eq:au_discrete} cost $\calO(dn^2r^2)$ and $\calO(dnR^2r^2)$, respectively,
and the dominant contribution comes from the quartic term: the Hadamard products
$\tU \odot \tW$ and $\tU \odot \tZ$ raise the rank to $2r^2$, and contracting
them costs $\calO(dnr^6)$. To avoid this growth, we round the Hadamard products
back to rank $\ur$, at the $\calO(dnr^3)$ cost of \eqref{eq:retraction}, so that
the evaluations of $\funch$ and of $\matauh \tZ$ both cost $\calO(dn^2r^2 +
    dnR^2r^2 + dnr^4)$. In every case the cost is linear in $d$ and in $n$, against
the $\calO(n^d)$ of the full format.

\begin{remark}[Evaluating the energy with a higher-order quadrature rule]
    \label{rmk:highorder_quadrature}
    Employing the quadrature formula that uses the GLL points introduces an
    additional source of error, justified by the computational savings discussed
    above. At the end of the optimization process, however, the energy can be
    evaluated with a higher-order quadrature at moderate additional cost.
    For example, exact integration of the
    nonlinearity requires increasing the quadrature order from $2k-1$ to $4k$. This
    is achieved by introducing the extension operator mapping $\tU$, the tensor of
    coefficients in the rescaled basis \eqref{eq:rescaled-basis}, to $\tU_Q$, the
    tensor of values at quadrature nodes (for order $4k$). By
    \eqref{eq:coeff-scaling}, this operator first undoes the rescaling, mapping
    $\tU$ to the tensor of nodal values by a Hadamard product with
    $\masshtenssqrtrec = \otimes_{i=1}^d [m_1^{-1/2}, \cdots, m_n^{-1/2}]^\top$, and
    then evaluates at the quadrature nodes. The energy is then computed using
    \eqref{eq:E_discrete} with $\tU_Q$ and the updated quadrature points and
    weights, the latter now multiplying the nonlinear terms rather than appearing as
    reciprocals. Since both factors of the extension operator are tensor products of
    their one-dimensional counterparts, it is a rank-$1$ TT matrix, so $\tU_Q$ has
    the same TT rank as $\tU$ and the extension costs $\calO(dn^2r^2)$, still polynomial in $n$ but quadratic instead of linear.
\end{remark}

\section{Preconditioning}
\label{sec:precond}

As discussed above, after discretization, computing the Riemannian gradient
reduces to solving linear systems with $\PAhPMr \coloneqq \Projvecm_{\tU}
    \matauh \Projvecm_{\tU}$ via the conjugate gradient method. If $\matauh$ is
ill-conditioned, one can expect $\PAhPMr$ to inherit poor conditioning. In this
section, we detail the construction of efficient preconditioners for $\matauh$
and $\PAhPMr$.
\subsection{Preconditioning in the full tensor format}

\label{sec:precond_full}

Since $\matauh$ has the form \eqref{eq:mat_au_discrete}, a natural idea is to
precondition it with the stiffness matrix $\stiffh$. This is particularly
appealing because $\stiffh^{-1}$ can be applied efficiently using the Fast
Diagonalization method \cite{lynch_direct_1964}. Since the mass matrix is the
identity in the basis \eqref{eq:rescaled-basis}, this only requires the
eigendecomposition of the symmetric matrix $\stiffoned$,
\begin{equation*}
    \mtW^\top \stiffoned \mtW = \mtLambda = \diag(\vlambda),
    \qquad
    \mtW^\top \mtW = \mtI,
\end{equation*}
with $\mtW$ orthogonal.
Then $\stiffh$ admits the decomposition
\begin{equation*}
    \stiffh = \left(\mtW \otimes \cdots \otimes \mtW\right)
    \mtD
    \left(\mtW \otimes \cdots \otimes \mtW\right)^{\top},
    \quad
    \mtD \coloneqq
    \mtLambda \otimes \mtI \otimes \cdots \otimes \mtI
    + \dots +
    \mtI \otimes \cdots \otimes \mtI \otimes \mtLambda,
\end{equation*}
with inverse
\begin{equation*}
    \stiffh^{-1} = \left(\mtW \otimes \cdots \otimes \mtW\right)
    \mtD^{-1}
    \left(\mtW \otimes \cdots \otimes \mtW\right)^{\top}.
\end{equation*}
Here $\mtD = \diag(\tD)$ with $\tD(i_1,\dots,i_d) = \lambda_{i_1} + \cdots +
    \lambda_{i_d}$, so $\mtD$ is diagonal and easily inverted. The setup requires a
single $n \times n$ eigendecomposition, costing $\calO(n^3)$ operations, and each application of
$\stiffh^{-1}$ amounts to $d$ mode-wise products with $\mtW$ and $\mtW^\top$, at
a cost of $\calO(dn^{d+1})$ operations, instead of the $\calO(n^{3d})$ of a direct solve.

A more effective preconditioner is obtained by also including an approximation
of the potential term $\int_{\Omega} V w z$. We approximate $V$ by a potential
$V_0 \approx V$ of the form
\begin{equation}
    \label{eq:V0}
    V_0(x_1,\dots,x_d) = v_1(x_1) + \dots + v_d(x_d) \geq 0.
\end{equation}
Then
\[
    \innerquad{V_0 w}{z}
    = \innersmall{\tV_0 \odot \tW}{\tZ}
    = \innersmall{\mtV_0 \tW}{\tZ},
\]
where $\tV_0(i_1,\dots,i_d) = V_0(x_{i_1},\dots,x_{i_d})$, i.e.,
\[
    \tV_0 = \vv_1 \otimes \ve \otimes \cdots \otimes \ve
    + \dots +
    \ve \otimes \cdots \otimes \ve \otimes \vv_d,
    \qquad \vv_i = [v_i(x_1),\dots,v_i(x_n)]^\top,
\]
with $\ve$ the vector of all ones, and the corresponding TT matrix is
\begin{equation*}
    \mtV_0 = \diag(\tV_0) =
    \mtV_1^{\oned} \otimes \mtI \otimes \cdots \otimes \mtI
    + \dots +
    \mtI \otimes \cdots \otimes \mtI \otimes \mtV_d^{\oned},
    \qquad
    \mtV_i^{\oned} = \diag(\vv_i).
\end{equation*}
We thus precondition with
\begin{equation}
    \label{eq:Btens}
    \precon = \stiffh + \mtV_0
    = (\stiffoned + \mtV_1^{\oned}) \otimes \mtI \otimes \cdots \otimes \mtI
    + \dots
    + \mtI \otimes \cdots \otimes \mtI \otimes (\stiffoned + \mtV_d^{\oned}),
\end{equation}
which can be inverted by Fast Diagonalization. The eigendecompositions
\begin{equation*}
    \mtW_i^\top (\stiffoned + \mtV_i^{\oned}) \mtW_i = \mtLambda_i = \diag(\vlambda_i),
    \quad
    \mtW_i^\top \mtW_i = \mtI,
    \qquad i=1,\dots,d,
\end{equation*}
yield the factorization
\begin{equation*}
    \precon = \left(\mtW_1 \otimes \cdots \otimes \mtW_d\right)
    \mtK
    \left(\mtW_1 \otimes \cdots \otimes \mtW_d\right)^{\top},
    \quad
    \mtK \coloneqq
    \mtLambda_1 \otimes \mtI \otimes \cdots \otimes \mtI
    + \dots +
    \mtI \otimes \cdots \otimes \mtI \otimes \mtLambda_d,
\end{equation*}
with inverse
\begin{equation}
    \label{eq:preconfull_inv}
    \precon^{-1} = \left(\mtW_1 \otimes \cdots \otimes \mtW_d\right)
    \mtK^{-1}
    \left(\mtW_1 \otimes \cdots \otimes \mtW_d\right)^{\top}.
\end{equation}
Note that since we assumed $V_0 \geq 0$, the preconditioner $\precon$ is
symmetric positive definite and can be used in the conjugate gradient method.

\begin{remark}
    While the stiffness matrix alone may appear to be a reasonable
    preconditioner, numerical experiments show that it is not effective. This
    is consistent with
    \eqref{eq:au_coercivity}, since the continuity constant grows with
    $\norm{u}_{L^{\infty}(\Omega)}$. Incorporating an approximation of the
    potential empirically improves performance, but the continuity constant of
    $a_u$ with respect to the norm induced by $\int_{\Omega} \nabla v \cdot
        \nabla w + \int_{\Omega} V_0 v w$ still depends on
    $\norm{u}_{L^{\infty}(\Omega)}$. A preconditioner that also includes an
    approximation of the discretization of $\int_{\Omega} u^2 v z$ could further
    enhance convergence.
\end{remark}
\subsection{Preconditioning in the TT format}
\label{sec:precond_tt}
We now address the preconditioning of $\PAhPMr = \Projvecm_{\tU} \matauh
    \Projvecm_{\tU}$. If $\precon$ is an effective preconditioner for $\matauh$,
it is natural to consider $\Pprecon \coloneqq \Projvecm_{\tU} \precon
    \Projvecm_{\tU}$ as a preconditioner for $\PAhPMr$. However, applying
$\Pprecon^{-1}$ to a tangent vector $\txi \in \Tang_{\tU}\vecmanif$, i.e.,
solving
\[
    (\Projvecm_{\tU} \precon \Projvecm_{\tU}) \teta = \txi,
    \qquad \teta \in \Tang_{\tU}\vecmanif,
\]
is computationally demanding; see \cite{kressner_preconditioned_2016,bioli_preconditioned_2025}. To reduce
the complexity, we proceed in two steps.

First, we replace $\Pprecon^{-1}$ by
\[
    \Ppreconinv \coloneqq \Projvecm_{\tU}\precon^{-1}\Projvecm_{\tU}.
\]
Tangent vectors in $\Tang_{\tU}\vecmanif = \range(\Projvecm_{\tU})$ are TT tensors of rank at most
$2\ur$, but the application of $\precon^{-1}$ from \eqref{eq:preconfull_inv} is not
well-suited to the TT format. The difficulty lies in $\mtK^{-1}$, which
corresponds to the Hadamard product with the tensor $\tKinv \coloneqq 1 / \tK$,
the entry-wise reciprocal of
\[
    \tK = \vlambda_1 \otimes \ve \otimes \cdots \otimes \ve
    + \dots +
    \ve \otimes \cdots \otimes \ve \otimes \vlambda_d,
    \qquad
    \tK(i_1, \dots, i_d) = (\vlambda_1)_{i_1} + \dots + (\vlambda_d)_{i_d}.
\]
While $\tK$ has TT rank $2$, enabling efficient Hadamard products with tensors
in TT format, the same is not true for $\tKinv$. Indeed, the entry-wise
reciprocal of a TT tensor has in general full TT rank, so that forming and
applying $\tKinv$ would cost $\calO(n^d)$ and destroy the computational advantages of the TT format.

As a second step, we approximate $\precon^{-1}$ by a symmetric positive definite
operator obtained by replacing $\tKinv$ with an approximation based on
exponential sums. Following \cite{montardini_low-rank_2023}, we approximate $g(\lambda) =
    1/\lambda$ by an exponential sum and apply it entry-wise to $\tK$. We recall the
following result.

\begin{proposition}[{\cite{kressner_krylov_2010,braess_approximation_2005}}]
    \label{prop:exp_sum}
    Let $s_k(\mu) = \sum_{i=1}^k \omega_i e^{-\alpha_i \mu}$, with
    $\alpha_i,\omega_i \in \R$. Then there exist $\alpha_i > 0$, $\omega_i > 0$
    (depending on $k$ and $R > 1$) such that
    \[
        \sup_{\mu \in [1,R]} \left| \frac{1}{\mu} - s_k(\mu) \right|
        \leq 16 \exp\!\left(-\tfrac{k\pi^2}{\log(8R)}\right).
    \]
\end{proposition}

Applying this approximation entry-wise to $\tK$ with the substitution
\[
    \mu = \frac{\tK(i_1, \dots, i_d)}{\min(\tK)} \in [1,R],
    \qquad
    R = \frac{\max(\tK)}{\min(\tK)},
\]
where
\begin{equation*}
    \begin{split}
        \min(\tK) & = \min_{i_1,\dots,i_d}\tK(i_1, \dots, i_d) = \min(\vlambda_1) + \dots + \min(\vlambda_d),
        \\
        \max(\tK) & = \max_{i_1,\dots,i_d}\tK(i_1, \dots, i_d) = \max(\vlambda_1) + \dots + \max(\vlambda_d),
    \end{split}
\end{equation*}
yields the approximation $\tKinvapp \approx \tKinv$:
\begin{equation}
    \label{eq:tKinvapp_def}
    \tKinvapp =
    \sum_{i=1}^{k} \tilde{\omega}_{i}\,
    e^{-\tilde{\alpha}_i \vlambda_1} \otimes \cdots \otimes e^{-\tilde{\alpha}_i \vlambda_d},
    \qquad
    \tilde{\alpha}_i = \tfrac{\alpha_i}{\min(\tK)}, \quad
    \tilde{\omega}_i = \tfrac{\omega_i}{\min(\tK)},
\end{equation}
with exponentials applied entry-wise. Tables of $(\alpha_i,\omega_i)$ for
various $R$ and $k$ are given in \cite{hackbusch_computation_2019}.

The resulting approximate inverse is
\[
    \Ppreconapp \coloneqq \Projvecm \preconapp \Projvecm,
    \qquad
    \preconapp = \big(\mtW_1 \otimes \cdots \otimes \mtW_d\big)
    \diag(\tKinvapp)
    \big(\mtW_1 \otimes \cdots \otimes \mtW_d\big)^\top.
\]
Applying $\mtW = \mtW_1 \otimes \cdots \otimes \mtW_d$ and $\mtW^\top$ does not
increase TT rank, since $\mtW$ is a rank-$1$ TT matrix. The Hadamard product with
$\tKinvapp$, on the other hand, increases the rank by a factor $k$, since
$\tKinvapp$ is a sum of $k$ rank-one terms. We therefore never form it, and
apply $\Ppreconapp$ instead in the split form
\begin{equation}
    \label{eq:precon_split}
    \Ppreconapp = \sum_{i=1}^k \tilde{\omega}_{i}
    \Projvecm \Big[\mtW\,
        \diag\big(e^{-\tilde{\alpha}_i \vlambda_1}
        \otimes \cdots \otimes e^{-\tilde{\alpha}_i \vlambda_d}\big)\,
        \mtW^\top\Big] \Projvecm,
\end{equation}
one term at a time. No intermediate quantity then exceeds TT rank $2\ur$: each
term is a Hadamard product with a rank-one tensor followed by rank-preserving
operations, and the $k$ results, being tangent vectors at the same base point
$\tU$, are accumulated by summing their parametrizations $\delta G_j$. The cost
is $\calO(k\,dn^2r^2 + k\,dnr^3)$ FLOPs and $\calO(dnr^2)$ memory, respectively
a factor $k$ and $k^2$ less than forming the sum and projecting it back at once,
since both the application of $\mtW$ and the projection are quadratic in the
rank of their input. In practice, we observed that a moderate $k$ (about $k=10$)
suffices to obtain an effective preconditioner.

\begin{remark}
    Throughout this remark, the inverse is understood as the inverse on
    $\Tang_{\tU}\vecmanif$. In general,
    \[
        \P_{\precon}^{-1}
        \coloneqq
        (\Projvecm_{\tU} \precon \Projvecm_{\tU})^{-1}
        \neq
        \Projvecm_{\tU} \precon^{-1} \Projvecm_{\tU}
        \eqqcolon \P_{\precon^{-1}},
    \]
    unless the projection is $\precon$-orthogonal, which is not the case here.
    Using $\P_{\precon}$ to precondition $\PAhPMr$ has the advantage that the
    spectrum of $\P_{\precon}^{-1}\PAhPMr$ interlaces with that of
    $\precon^{-1}\matauh$ \cite{lancaster_variational_1991}. Consequently,
    $\kappa\big(\P_{\precon}^{-1}\PAhPMr\big) \leq
        \kappa\big(\precon^{-1}\matauh\big)$. In contrast, we are not aware of any analogous spectral relation for $\P_{\precon^{-1}}\PAhPMr$.
\end{remark}

\section{The multicomponent Gross--Pitaevskii equation}
\label{sec:multicomponent}

We now extend the framework of the previous sections to the multicomponent
Gross--Pitaevskii equation. The single-component case is recovered for $p=1$,
and we present the multicomponent case separately because most of the
difficulties it raises are best appreciated once the single-component
construction is understood. Throughout this section we follow
\cite{altmann_riemannian_2025} for the continuous formulation.

\subsection{Problem formulation}
\label{sec:multi-formulation}

A stationary state of a $p$-component Bose--Einstein condensate is modelled by a
$p$-frame $\phi = (\phi_1, \dots, \phi_p)$, $\phi_j\colon \Omega \to \R$, whose
components carry the prescribed masses
\begin{equation}
    \label{eq:multi-mass}
    \norm{\phi_j}_{L^2(\Omega)}^2 = N_j > 0, \qquad j = 1, \dots, p.
\end{equation}
The interaction between the components is encoded in the densities
\begin{equation}
    \label{eq:multi-density}
    \rho_j(\phi) = \sum_{i=1}^p \kappa_{ij} \abs{\phi_i}^2, \qquad j = 1, \dots, p,
\end{equation}
where the interaction matrix $\mK = [\kappa_{ij}]_{i,j=1}^p \in \R^{p\times p}$
is symmetric positive definite \cite[Assumption~A2]{altmann_riemannian_2025}. We
also assume that $\kappa_{ij}\geq 0$ for all $i$ and $j$, so that all bilinear
forms associated with the Fréchet derivative of the energy define a scalar
product (see \eqref{eq:multi-au-Eprime} and the discussion after that).
The energy functional reads
\begin{equation}
    \label{eq:multi-energy}
    \func(\phi) = \sum_{j=1}^p \int_{\Omega}
    \frac 12 \abs{\nabla \phi_j}^2 +
    \frac 12 V_j(x) \abs{\phi_j}^2 +
    \frac{\beta}{2} \rho_j(\phi) \abs{\phi_j}^2 \,dx,
\end{equation}
with external potentials $V_1, \dots, V_p \in L^{\infty}(\Omega)$ with $V_j \geq
    0$ a.e. in $\Omega$ \cite[Assumption~A1]{altmann_riemannian_2025}.
Minimizers of \eqref{eq:multi-energy} subject to \eqref{eq:multi-mass} are the
ground states, and they solve the coupled Gross--Pitaevskii equations
\[
    -\Delta \phi_j + V_j \phi_j + 2\beta\rho_j(\phi)\phi_j = \lambda_j \phi_j,
    \qquad j = 1, \dots, p,
\]
a nonlinear eigenvector problem whose eigenvalues are the chemical potentials of
the components.

\paragraph{Functional setting and the generalized oblique manifold.}
We work in the Hilbert spaces $\Lspace = [L^2(\Omega)]^p$ and $\Hspace =
    [H_0^1(\Omega)]^p$, and we introduce the diagonal matrix of component-wise
$L^2$ inner products
\[
    \dinner{v}{w} = \diag\bigl( (v_1, w_1)_{L^2(\Omega)}, \dots, (v_p, w_p)_{L^2(\Omega)} \bigr) \in \Dp,
\]
where $\Dp$ denotes the set of $p\times p$ real diagonal matrices. We set $\mN =
    \diag(N_1, \dots, N_p)$. The mass constraints \eqref{eq:multi-mass} then read
$\dinner{\phi}{\phi} = \mN$, and the admissible states form the
\emph{generalized oblique manifold}
\begin{equation}
    \label{eq:oblique}
    \OB = \left\{ \phi \in \Hspace : \dinner{\phi}{\phi} = \mN \right\}
    = \sphere_{N_1} \times \dots \times \sphere_{N_p},
    \qquad
    \sphere_{N_j} = \left\{ \phi \in H^1_0(\Omega) : \norm{\phi}^2_{L^2(\Omega)} = N_j \right\},
\end{equation}
a closed embedded submanifold of $\Hspace$ of codimension $p$, with tangent space
\[
    \Tang_\phi \OB = \left\{ z \in \Hspace : \dinner{\phi}{z} = \mathbf{0} \right\}.
\]
The manifold $\OB$ replaces the $L^2$-sphere $\sphere$ of the single-component
case, to which it reduces for $p = 1$.

The energy \eqref{eq:multi-energy} is Fréchet differentiable on $\Hspace$ with
\begin{equation}
    \label{eq:multi-au-Eprime}
    \func'(\phi)[w] = a_\phi(\phi, w), \qquad
    a_\phi(v,w) = \sum_{j=1}^p \int_{\Omega} \nabla v_j \cdot \nabla w_j
    + V_j(x)\, v_j w_j + 2\beta \rho_j(\phi)\, v_j w_j \,dx.
\end{equation}
Since $\kappa_{ij}\geq 0$ for all $i,j$, then $a_\phi$ is continuous and
coercive on $\Hspace$ and \eqref{eq:multi-au-Eprime} defines, as in
\eqref{eq:metric}, the energy-adaptive metric
$\innersmall[\phi]{v}{w} = a_\phi(v,w)$.
The associated Gross--Pitaevskii Hamiltonian $\Aop_\phi\colon \Hspace \to
    \Hspace'$ is defined by $\innersmall{\Aop_\phi v}{w} = a_\phi(v,w)$. It acts
component-wise,
\begin{equation}
    \label{eq:Aop-componentwise}
    \Aop_\phi v = (\Aop_{\phi,1} v_1, \dots, \Aop_{\phi,p} v_p),
    \qquad
    \innersmall{\Aop_{\phi,j} v_j}{w_j} = \int_{\Omega} \nabla v_j \cdot \nabla w_j
    + V_j v_j w_j + 2\beta \rho_j(\phi) v_j w_j \,dx,
\end{equation}
and consequently $\Aop_\phi(v\Lambda) = (\Aop_\phi v)\Lambda$ for every diagonal
matrix $\Lambda \in \Dp$. This component-wise structure is the source of all the
computational simplifications available on $\OB$, and, as we will see, it may
be lost in a naive application of TT-rank compression, as discussed in
\Cref{sec:multi-grad-Hu}.

\paragraph{Geometry of the oblique manifold.}
For later reference we recall the geometry of $\OB$ from
\cite{altmann_riemannian_2025}. Given a metric $g_\phi$ on $\Hspace$ with
associated operator $\Gop_\phi \colon \Hspace \to \Hspace'$,
$\innersmall{\Gop_\phi v}{w} = g_\phi(v,w)$, the $g_\phi$-orthogonal projection
onto $\Tang_\phi\OB$ is
\begin{equation}
    \label{eq:proj-OB}
    \Projob_\phi(u) = u - \Gop_\phi^{-1}\bigl(\phi\, \Sigma_{\phi,u}\bigr),
    \qquad
    \dinner{\phi}{\Gop_\phi^{-1}(\phi\,\Sigma_{\phi,u})} = \dinner{\phi}{u},
\end{equation}
where $\Sigma_{\phi,u} \in \Dp$ is uniquely determined
\cite[Proposition 9]{altmann_riemannian_2025}. Writing $\Sigma_{\phi,u} =
    \diag(\vsigma)$ and introducing the $j$-th frame
\begin{equation}
    \label{eq:frame}
    \hphi_j = (0, \dots, 0, \phi_j, 0, \dots, 0) \in \Hspace,
\end{equation}
the second identity in \eqref{eq:proj-OB} is the $p \times p$ linear system
$\mG\vsigma = \vb$ with
\begin{equation}
    \label{eq:Gsystem}
    \mG[:,j] = \diag \dinner{\Gop_\phi^{-1}\hphi_j}{\phi},
    \qquad
    \vb = \diag \dinner{\phi}{u}.
\end{equation}
Assembling $\mG$ requires $p$ applications of $\Gop_\phi^{-1}$. If, however, the
metric factorizes over the components, as $a_\phi$ does, by
\eqref{eq:Aop-componentwise}, then $\mG$ is diagonal and
\cite[Corollary 10]{altmann_riemannian_2025}
\begin{equation}
    \label{eq:Sigma-diagonal}
    \Sigma_{\phi,u} = \dinner{\phi}{\Gop_\phi^{-1}\phi}^{-1}\dinner{\phi}{u},
\end{equation}
so that a single solve suffices. In the energy-adaptive metric this yields
\begin{equation}
    \label{eq:grad-OB}
    \grad \func(\phi) = \Projob_\phi(\phi) = \phi - \Aop_\phi^{-1}\bigl(\phi\,\Sigma_{\phi,\phi}\bigr),
\end{equation}
and a retraction onto $\OB$ is given by component-wise normalization,
\begin{equation}
    \label{eq:retraction-OB}
    \Retr(v) = v \dinner{v}{v}^{-1/2}\mN^{1/2}, \qquad v \in \Hspace.
\end{equation}

\subsection{Discretization}
\label{sec:multi-discretization}

We discretize each component in the same tensor product space $\X_N$ of
\Cref{sec:numerical_discretization}, with the rescaled SEM basis
\eqref{eq:rescaled-basis}. A discrete $p$-frame is then represented by a
$(d+1)$-dimensional tensor
\[
    \tPhi \in \R^{n \times \dots \times n \times p},
    \qquad
    \phi(x) = \sum_{i_1, \dots, i_d = 1}^n \tPhi(i_1, \dots, i_d, :)\,
    \tphioned_{i_1}(x_1) \cdots \tphioned_{i_d}(x_d),
\]
the last mode indexing the components. We write $\tPhi_j = \tPhi(:, \dots, :, j)
    \in \Rnd$ for the coefficient tensor of $\phi_j$. As in
\eqref{eq:coeff-scaling}, the entries of $\tPhi$ carry the square roots of the
quadrature weights,
\[
    \tPhi(\is, j) = \sqrt{m_{i_1}\cdots m_{i_d}}\; \phi_j(x_{i_1}, \dots, x_{i_d}),
\]
and consequently the mass matrix is again the identity, $\masshp = \mtI$, while
the weights reappear in the terms that are nonlinear in $\tPhi$.

All discrete operators of \Cref{sec:numerical_discretization} extend to the
multicomponent setting by appending a single core acting on the component mode:
\begin{equation}
    \label{eq:multi-operators}
    \stiffhp = \stiffh \otimes \mtI_p,
    \qquad
    \masshtensrecp = \masshtensrec \otimes \ve_p,
    \qquad
    \couplp = \mtI \otimes \dots \otimes \mtI \otimes \mK,
\end{equation}
where $\ve_p \in \R^p$ is the vector of all ones. Since $\stiffh$ has TT rank $2$
and $\masshtensrec$ is a rank-$1$ TT tensor \eqref{eq:Mrec}, the appended core
does not increase the TT ranks: $\stiffhp$ has TT rank $2$, $\masshtensrecp$ is a
rank-$1$ TT tensor, and $\couplp$ is a rank-$1$ TT matrix acting as the identity
on all spatial modes and whose last core is the interaction matrix $\mK$ itself.
In other words, the whole coupling between the components is carried by one $p
    \times p$ core.

We collect the potentials in the tensor $\tV_p(\is, j) = V_j(x_{i_1}, \dots,
    x_{i_d})$, assumed to be of low TT rank. If, as in all our numerical
experiments, all components share the same potential $V_1 = \dots = V_p = V$,
then $\tV_p = \tV \otimes \ve_p$ has the same TT rank as $\tV$.
The discrete energy and bilinear form are then the exact counterparts of
\eqref{eq:E_discrete} and \eqref{eq:au_discrete},
\begin{equation}
    \label{eq:multi-E-discrete}
    \funch(\tPhi) =
    \frac 12 \innersmall{\tPhi}{\stiffhp \tPhi}
    + \frac 12 \innersmall{\tV_p \odot \tPhi}{\tPhi}
    + \frac{\beta}{2} \innersmall{\tPhi \odot \tPhi}{\masshtensrecp \odot \couplp (\tPhi \odot \tPhi)},
\end{equation}
\begin{equation}
    \label{eq:multi-au-discrete}
    \auhphi{\tU}{\tW} =
    \innersmall{\tU}{\stiffhp \tW}
    + \innersmall{\tV_p \odot \tU}{\tW}
    + 2\beta \innersmall{\tU \odot \tW}{\masshtensrecp \odot \couplp (\tPhi \odot \tPhi)},
\end{equation}
so that, explicitly,
\begin{equation}
    \label{eq:multi-mat_au_discrete}
    \matauhphi = \stiffhp + \diag(\tV_p) + 2\beta \diag\bigl(\masshtensrecp \odot \couplp(\tPhi \odot \tPhi)\bigr).
\end{equation}
Comparing with \eqref{eq:mat_au_discrete}, the only structural change is that the
pointwise density $\tU \odot \tU$ is replaced by $\couplp(\tPhi \odot \tPhi)$,
i.e., by the same density mixed across components by $\mK$. Every operation in
\eqref{eq:multi-E-discrete}--\eqref{eq:multi-mat_au_discrete} is therefore a TT
operation of exactly the same kind as in the single-component case, on a tensor
with one additional mode of size $p$; in particular, applying $\couplp$ costs one
$p \times p$ matrix product on the last core.

\subsection{The manifold of multicomponent unit-norm fixed TT-rank functions}
\label{sec:multi-manifold}

We keep the notation of \Cref{sec:manifolds}, with $\sphere$ replaced by $\OB$
and with $\matauh$, $\PAhPMr$ replaced by their multicomponent counterparts
$\matauhphi$ of \eqref{eq:multi-mat_au_discrete} and $\PAhPMrphi \coloneqq
    \Projvecm_{\tPhi}\,\matauhphi\,\Projvecm_{\tPhi}$.
Thus $\vecmanif$ denotes the set of TT tensors of fixed TT rank $\ur = (r_0,
    \dots, r_d, r_{d+1})$, $r_0 = r_{d+1} = 1$, in $\R^{n \times \dots \times n
        \times p}$, $\manif = \synthesis_N(\vecmanif)$ the corresponding set of
$p$-frames in $\X_N$, and
\begin{equation}
    \label{eq:multi-intersection}
    \cmanif = \manif \cap \OBN .
\end{equation}
At the tensor level we correspondingly write
\begin{equation}
    \label{eq:multi-vecN}
    \vecOB = \bigl\{ \tPhi \in \R^{n \times \dots \times n \times p} : \dinner{\tPhi}{\tPhi} = \mN \bigr\},
    \qquad
    \veccmanif = \vecmanif \cap \vecOB = \analysis_N(\cmanif),
\end{equation}
which generalizes the set $\veccmanif$ of \Cref{sec:numerical_discretization}. We will
consistently use calligraphic manifold symbols $\manif, \OBN, \cmanif$ for subsets of $\X_N$
and Fraktur ones $\vecmanif, \vecOB, \veccmanif$ for the corresponding subsets of
$\R^{n \times \dots \times n \times p}$.

\begin{proposition}
    \label{prop:multi-intersection}
    $\cmanif$ is an embedded submanifold of $\X_N$, and for all $\phi \in \cmanif$
    \[
        \Tang_\phi \cmanif = \Tang_\phi \manif \cap \Tang_\phi \OBN .
    \]
\end{proposition}
\begin{proof}
    As in the single-component case it suffices to prove that $\manif$ and $\OBN$
    intersect transversally, i.e., that $\Tang_\phi \manif + \Tang_\phi \OBN =
        \X_N$. Since $\Tang_\phi\OBN$ has codimension $p$ and its complement is
    spanned by the frames $\hphi_1, \dots, \hphi_p$ of \eqref{eq:frame}, it is
    enough to show $\hphi_j \in \Tang_\phi\manif$ for all $j = 1, \dots, p$. This
    holds by choosing, in the parametrization of $\Tang_{\tPhi}\vecmanif$,
    \[
        \delta G_k = 0 \quad \text{for } k = 1, \dots, d,
        \qquad
        \delta G_{d+1}(i) = \delta_{ij}\, G_{d+1}(i).
    \]
    This choice also trivially satisfies the gauge conditions \eqref{eq:gauge}.
        Transversality then follows,
    and with it the claim, by~\cite[Corollary 3.5.13]{Abraham1988}.
\end{proof}

For $\phi \in \cmanif$ and $u \in \X_N$ we denote by $\Projm_\phi$, $\Projob_\phi$ and
$\Projn_\phi$ the $a_\phi$-orthogonal projections onto $\Tang_\phi\manif$,
$\Tang_\phi\OBN$ and $\Tang_\phi\cmanif$, respectively. Exactly as in
\eqref{eq:Riemannian-gradient}, the Riemannian gradient of $\func$ on $\cmanif$ in the
$a_\phi$ metric is
\begin{equation}
    \label{eq:multi-grad}
    \grad \func(\phi) = \Projn_\phi (\phi),
\end{equation}
since $\Projn_\phi(\phi) \in \Tang_\phi\cmanif$ and, for all $v \in \Tang_\phi\cmanif$,
\[
    \func'(\phi)[v] = a_\phi(\phi, v) = a_\phi\bigl(\Projn_\phi(\phi), v\bigr),
\]
by $a_\phi$-orthogonality of the projection.

The remainder of this section is devoted to the computation of
\eqref{eq:multi-grad}. We present two routes. The first mirrors the continuous
construction \eqref{eq:proj-OB}--\eqref{eq:grad-OB} and is theoretically
transparent, but expensive; the second is the one we use in practice.

\subsection{Riemannian gradient I: the operator \texorpdfstring{$\Hop_\phi$}{Hphi}}
\label{sec:multi-grad-Hu}

As in \Cref{def:Hop}, we define $\Hop_\phi \colon \X_N \to \Tang_\phi \manif$ by
\begin{equation}
    \label{eq:multi-Hop}
    a_\phi(\Hop_\phi v, w) = (v, w)_{\Lspace}, \qquad \forall w \in \Tang_\phi\manif,
\end{equation}
which is well defined by coercivity and continuity of $a_\phi$ and the
Lax--Milgram lemma. The following is the multicomponent analogue of
\Cref{prop:projN}.

\begin{proposition}
    \label{prop:multi-projN}
    Let $\phi \in \cmanif$ and $u \in \X_N$. Then
    \begin{equation}
        \label{eq:multi-projN}
        \Projn_\phi u = \Projm_\phi u - \Hop_\phi\bigl(\phi\, \Sigma_{\phi,u}\bigr),
        \qquad \Sigma_{\phi,u} = \diag(\vsigma_{\phi,u}),
    \end{equation}
    where $\vsigma_{\phi,u} \in \R^p$ is the unique solution of the linear system
    $\mG \vsigma = \vb$ with
    \begin{equation}
        \label{eq:multi-Gsystem}
        \mG[:,j] = \diag \dinner{\Hop_\phi \hphi_j}{\phi},
        \qquad
        \vb = \diag \dinner{\Projm_\phi u}{\phi}.
    \end{equation}
\end{proposition}
\begin{proof}
    \textit{Step 1.} Both $\Projm_\phi u$ and $\Hop_\phi(\phi\Sigma_{\phi,u})$
    belong to $\Tang_\phi\manif$, hence so does $\Projn_\phi u$.
    \\
    \textit{Step 2.} For all $w \in \Tang_\phi\cmanif = \Tang_\phi\manif \cap
        \Tang_\phi\OBN$,
    \[
        a_\phi(\Projn_\phi u, w)
        = a_\phi(\Projm_\phi u, w) - a_\phi\bigl(\Hop_\phi(\phi\Sigma_{\phi,u}), w\bigr)
        = a_\phi(u, w) - (\phi\Sigma_{\phi,u}, w)_{\Lspace}
        = a_\phi(u,w),
    \]
    where the second equality uses $w \in \Tang_\phi\manif$ together with
    \eqref{eq:multi-Hop}, and the last one follows from
    \[
        (\phi\Sigma_{\phi,u}, w)_{\Lspace} = \sum_{i=1}^p [\vsigma_{\phi,u}]_i\, (\phi_i, w_i)_{L^2(\Omega)} = 0,
        \qquad w \in \Tang_\phi \OBN .
    \]
    \\
    \textit{Step 3.} It remains to choose $\vsigma_{\phi,u}$ so that $\Projn_\phi
        u \in \Tang_\phi\OBN$. By linearity of $\Hop_\phi$ and
    \eqref{eq:frame}, $\Hop_\phi(\phi\Sigma_{\phi,u}) = \sum_{j}
        [\vsigma_{\phi,u}]_j \Hop_\phi\hphi_j$, so that
    \[
        \dinner{\Projn_\phi u}{\phi}
        = \dinner{\Projm_\phi u}{\phi} - \sum_{j=1}^p [\vsigma_{\phi,u}]_j \dinner{\Hop_\phi \hphi_j}{\phi}
        = \mathbf{0}
    \]
    is exactly the system \eqref{eq:multi-Gsystem}. Its unique solvability is the
    content of \Cref{lem:multi-G-invertible} below.
\end{proof}

\begin{lemma}
    \label{lem:multi-G-invertible}
    The matrix $\mG \in \R^{p\times p}$ in \eqref{eq:multi-Gsystem} is invertible.
\end{lemma}
\begin{proof}
    Let $\valpha \in \R^p$ be such that $\mG\valpha = 0$. Then, for all $i = 1,
        \dots, p$,
    \[
        0 = \sum_{j=1}^p \alpha_j \bigl( (\Hop_\phi \hphi_j)_i, \phi_i \bigr)_{L^2(\Omega)}
        = \bigl( (\Hop_\phi(\phi \diag \valpha))_i, \phi_i \bigr)_{L^2(\Omega)},
    \]
    that is, $\dinner{\Hop_\phi(\phi\diag\valpha)}{\phi} = \mathbf{0}$, so that
    $\Hop_\phi(\phi\diag\valpha) \in \Tang_\phi\OBN$. On the other hand, for every
    $w \in \Tang_\phi\manif \cap \OBN$,
    \[
        a_\phi\bigl(\Hop_\phi(\phi\diag\valpha), w\bigr)
        = (\phi\diag\valpha, w)_{\Lspace}
        = \sum_{i=1}^p \alpha_i (\phi_i, w_i)_{L^2(\Omega)} = 0 .
    \]
    Since $\Hop_\phi(\phi\diag\valpha) \in \Tang_\phi\manif \cap \Tang_\phi\OBN$,
    taking $w = \Hop_\phi(\phi\diag\valpha)$ and using the coercivity of $a_\phi$
    gives $\Hop_\phi(\phi\diag\valpha) = 0$. Finally, $\phi\diag\valpha \in
        \Tang_\phi\manif$ by the proof of \Cref{prop:multi-intersection}, hence
    \[
        0 = a_\phi\bigl(\Hop_\phi(\phi\diag\valpha), \phi\diag\valpha\bigr)
        = (\phi\diag\valpha, \phi\diag\valpha)_{\Lspace}
        = \sum_{i=1}^p \alpha_i^2 \norm{\phi_i}^2_{L^2(\Omega)}
        = \sum_{i=1}^p \alpha_i^2 N_i,
    \]
    which forces $\valpha = 0$.
\end{proof}

Taking $u = \phi$ in \Cref{prop:multi-projN} and using $\Projm_\phi\phi = \phi$
yields the multicomponent counterpart of \eqref{eq:grad-OB},
\begin{equation}
    \label{eq:multi-grad-Hu}
    \grad \func(\phi) = \phi - \Hop_\phi\bigl(\phi\,\Sigma_{\phi,\phi}\bigr),
    \qquad \vb = \diag\dinner{\phi}{\phi} = [N_1, \dots, N_p]^\top .
\end{equation}

\paragraph{Cost, and why the TT format is the obstruction.}
Formula \eqref{eq:multi-grad-Hu} is the exact analogue of the continuous
expression \eqref{eq:grad-OB}, and it is the formula one would use in the full
tensor format. There, however, it is much cheaper than it looks: by
\eqref{eq:Aop-componentwise} the Hamiltonian $\Aop_\phi$ is block diagonal with
respect to the component index, so the metric $a_\phi$ factorizes over the
components, the matrix $\mG$ of \eqref{eq:Gsystem} is diagonal, and
\eqref{eq:Sigma-diagonal} gives $\Sigma_{\phi,\phi}$ in closed form after \emph{a
    single linear solve} with $\matauh$.

This simplification is lost on $\cmanif$. The tangent space $\Tang_\phi\manif$ does
not split over the components: a TT tensor of rank $\ur$ ties all $p$
components to the same cores $G_1, \dots, G_d$, and only the last core
distinguishes them. Consequently $\Hop_\phi$ does not act component-wise,
$\mG$ in \eqref{eq:multi-Gsystem} is full, and assembling it requires applying
$\Hop_\phi$ once per component. Each such application is, by
\Cref{prop:projau_Hu}, a preconditioned CG solve with $\PAhPMr$ on
$\Tang_{\tPhi}\vecmanif$. The gradient \eqref{eq:multi-grad-Hu} therefore costs
$p+1$ such solves ($p$ for the columns of $\mG$, one for
$\Hop_\phi(\phi\Sigma_{\phi,\phi})$) against the single solve of the full
format. The extra cost is thus imposed entirely by the low-rank format, and not
by the multicomponent nature of the problem.

\subsection{Riemannian gradient II: change of metric}
\label{sec:multi-grad-PAPN}

The alternative to the strategy presented in the previous section is to bypass $\Hop_\phi$ altogether: compute the Riemannian
gradient of $\func$ with respect to a second metric $g$, for which it is
directly available, and then convert it into the gradient with the metric
induced by $a_\phi$. We describe
the construction first in the continuous setting, where it is
only formal for the metric we are interested in, and then in its discrete
counterpart, which is the one we implement and is well posed.

\paragraph{The continuous change of metric.}
Let $g_\phi$ be a metric on $\Hspace$ and let $\Projng{g}_\phi$ denote the
$g_\phi$-orthogonal projection onto $\Tang_\phi\cmanif$. The two gradients
$\grad\func(\phi)$ and $\grad[g]\func(\phi)$ represent the same functional,
namely the restriction of $\func'(\phi)$ to $\Tang_\phi\cmanif$, the first in the
$a_\phi$ inner product and the second in $g_\phi$:
\[
    a_\phi\bigl(\grad\func(\phi), v\bigr) = \func'(\phi)[v]
    = g_\phi\bigl(\grad[g]\func(\phi), v\bigr),
    \qquad \forall v \in \Tang_\phi\cmanif .
\]
Since $\grad\func(\phi) \in \Tang_\phi\cmanif$, this is equivalent to the operator
equation on $\Tang_\phi\cmanif$
\begin{equation}
    \label{eq:change-of-metric}
    \bigl( \Projng{g}_\phi\, \Aop_\phi\, \Projng{g}_\phi \bigr) \grad \func(\phi)
    = \grad[g] \func(\phi).
\end{equation}
Whenever $\grad[g]\func(\phi)$ and $\Projng{g}_\phi$ are cheaper to evaluate than
the $a_\phi$-orthogonal projection $\Projn_\phi$ of \Cref{prop:multi-projN},
solving \eqref{eq:change-of-metric} iteratively is an attractive alternative to
\eqref{eq:multi-grad-Hu}.

We take for $g$ the $L^2$ metric of \cite[Section 4.3.1]{altmann_riemannian_2025},
\begin{equation}
    \label{eq:L2-metric}
    g_{\Lspace}(z,y) = (z,y)_{\Lspace}, \qquad z, y \in \Tang_\phi\OB,
\end{equation}
for two reasons. First, $g_{\Lspace}$ does not depend on $\phi$ and is of the additive
form \eqref{eq:oblique}, so that the corresponding orthogonal projection onto
$\Tang_\phi\OB$ is available in closed form
\cite[eq.~(4.9)]{altmann_riemannian_2025},
\begin{equation}
    \label{eq:proj-L2}
    \Projobl_\phi(u) = u - \phi \dinner{\phi}{u}\mN^{-1};
\end{equation}
in contrast with \eqref{eq:proj-OB}, no operator has to be inverted. Second, after discretization $g_{\Lspace}$ becomes the Frobenius inner product of
the coefficient tensors, which is the inner product in which automatic
differentiation natively returns gradients.

This choice comes at a price at the continuous level. The metric $g_{\Lspace}$ violates
the assumptions of \Cref{sec:multi-formulation}: it is not coercive with respect
to the $\Hspace$-norm, the tangent space $\Tang_\phi\OB$ is not complete for it,
and the associated operator $\Gop_L$ is not necessarily invertible.
Consequently the $L^2$ Riemannian gradient need not exist for all $\phi \in
    \Hspace$, and neither does the $g_{\Lspace}$-orthogonal projection onto
$\Tang_\phi\cmanif$. Following \cite[Section 4.3.1]{altmann_riemannian_2025}, we
assume for the rest of this paragraph that $\phi$ is such that $\Aop_\phi \phi
    \in \Lspace$, so that all formulas are well defined. Under this assumption
$\func'(\phi) = \Aop_\phi\phi$ can be identified with an element of $\Lspace$,
the $L^2$ Riemannian gradient of $\func$ on $\cmanif$ is
\begin{equation}
    \label{eq:grad-L2-cont}
    \grad[L^2]\func(\phi) = \Projnl_\phi\bigl(\Aop_\phi \phi\bigr),
\end{equation}
and \eqref{eq:change-of-metric} reads
\begin{equation}
    \label{eq:change-of-metric-L2}
    \bigl( \Projnl_\phi\, \Aop_\phi\, \Projnl_\phi \bigr) \grad \func(\phi)
    = \grad[L^2]\func(\phi).
\end{equation}
Note that $\Projnl_\phi$ is $L^2$-orthogonal, whereas the projections
$\Projm_\phi$, $\Projob_\phi$, $\Projn_\phi$ of
\Cref{sec:multi-manifold,sec:multi-grad-Hu} are $a_\phi$-orthogonal; only the
operator $\Aop_\phi$ sandwiched in \eqref{eq:change-of-metric-L2} carries the
energy-adaptive metric.

\paragraph{The discrete counterpart.}
After discretization the difficulty disappears. By \eqref{eq:coeff-scaling} the
discrete $L^2$ inner product on $\X_N$ is the Frobenius inner product of the
coefficient tensors, which is positive definite on the finite-dimensional space
$\Tang_{\tPhi}\veccmanif$; the projections and the gradient
\eqref{eq:grad-L2-cont} therefore exist for every $\tPhi \in \veccmanif$, with no
assumption on $\tPhi$. We write
\[
    \Projvecm_{\tPhi}, \quad \Projvecob_{\tPhi}, \quad \Projvecn_{\tPhi}
    \colon \R^{n \times \dots \times n \times p} \longrightarrow
    \Tang_{\tPhi}\vecmanif, \quad \Tang_{\tPhi}\vecOB, \quad \Tang_{\tPhi}\veccmanif
\]
for the Frobenius-orthogonal projections onto the three tensor-level tangent
spaces, so that $\Projvecm_{\tPhi}$, $\Projvecob_{\tPhi}$ and $\Projvecn_{\tPhi}$
are the discrete counterparts of $\Projml_\phi$, $\Projobl_\phi$ and
$\Projnl_\phi$. Of these, only $\Projvecm_{\tPhi}$ has appeared so far, in
$\PAhPMrphi$ and in \Cref{prop:projau_Hu}: it is the standard TT tangent-space
projection. The tensor form of \eqref{eq:proj-L2} is
\begin{equation}
    \label{eq:proj-L2-disc}
    \Projvecob_{\tPhi} \tV = \tV - \tPhi \dinner{\tPhi}{\tV} \mN^{-1},
\end{equation}
and the three projections are related as follows.

\begin{lemma}
    \label{lem:proj-commute}
    For all $\tPhi \in \veccmanif$,
    $\Projvecn_{\tPhi} = \Projvecob_{\tPhi}\Projvecm_{\tPhi}
        = \Projvecm_{\tPhi}\Projvecob_{\tPhi}$.
\end{lemma}
\begin{proof}
    By \eqref{eq:proj-L2-disc} and \eqref{eq:frame}, $\Projvecob_{\tPhi}$ subtracts a
    linear combination of the frames $\hphi_1, \dots, \hphi_p$ of $\tPhi$, which
    are Frobenius-orthogonal with $\innersmall{\hphi_i}{\hphi_j} = N_i
        \delta_{ij}$. Hence $\Projvecob_{\tPhi}\tV = \tV - \sum_j
        N_j^{-1}\innersmall{\hphi_j}{\tV}\,\hphi_j$. By the proof of
    \Cref{prop:multi-intersection}, $\hphi_j \in \Tang_{\tPhi}\vecmanif$, so
    $\Projvecm_{\tPhi}\hphi_j = \hphi_j$ and
    $\innersmall{\hphi_j}{\Projvecm_{\tPhi}\tV} =
        \innersmall{\Projvecm_{\tPhi}\hphi_j}{\tV} = \innersmall{\hphi_j}{\tV}$.
    Both compositions therefore equal $\tV \mapsto \Projvecm_{\tPhi}\tV - \sum_j
        N_j^{-1}\innersmall{\hphi_j}{\tV}\hphi_j$, so
    $\Projvecm_{\tPhi}$ and $\Projvecob_{\tPhi}$ commute. Two commuting orthogonal
    projections compose to the orthogonal projection onto the intersection of
    their ranges, which by \Cref{prop:multi-intersection} is
    $\Tang_{\tPhi}\vecmanif \cap \Tang_{\tPhi}\vecOB = \Tang_{\tPhi}\veccmanif$.
\end{proof}

In particular $\Projvecn_{\tPhi}$ costs one application of $\Projvecm_{\tPhi}$
followed by a rank-$p$ correction, and by \Cref{lem:last-core} the correction
only touches the last core.

\paragraph{Computing the discrete Riemannian gradient.}
Differentiating $\funch$ with respect to the TT cores of $\tPhi$ by automatic
differentiation returns the Riemannian gradient on $\vecmanif$ with respect to
the Frobenius inner product \cite{novikov_automatic_2022}, which we denote by
$\gradF$ and refer to as the \emph{Frobenius Riemannian gradient}, as opposed to
the $a_\phi$-Riemannian gradient $\grad$ of \eqref{eq:multi-grad}:
\begin{equation}
    \label{eq:grad-F}
    \gradF\funch(\tPhi) = \Projvecm_{\tPhi}\bigl(\matauhphi \tPhi\bigr) \in \Tang_{\tPhi}\vecmanif,
    \qquad
    \funch'(\tPhi)[\tW] = \innersmall{\gradF\funch(\tPhi)}{\tW}
    \quad \forall \tW \in \Tang_{\tPhi}\vecmanif,
\end{equation}
where the identity follows from $\funch'(\tPhi)[\tW] = \auhphi{\tPhi}{\tW} =
    \innersmall{\matauhphi\tPhi}{\tW}$. This is a gradient on $\vecmanif$, not on
$\veccmanif$: the mass constraint has not yet been enforced. By
\Cref{lem:proj-commute} the Frobenius Riemannian gradient on $\veccmanif$ is obtained
by one further projection,
\begin{equation}
    \label{eq:grad-L2-N}
    \gradFN\funch(\tPhi) = \Projvecn_{\tPhi}\bigl(\matauhphi \tPhi\bigr)
    = \Projvecob_{\tPhi}\bigl(\gradF\funch(\tPhi)\bigr).
\end{equation}

Let $\tG = \analysis_N\bigl(\grad\func(\phi)\bigr) \in \Tang_{\tPhi}\veccmanif$ be the
tensor representation of the $a_\phi$-Riemannian gradient \eqref{eq:multi-grad}.
Both $\tG$ and $\gradFN\funch(\tPhi)$ represent the same functional
$\funch'(\tPhi)$ restricted to $\Tang_{\tPhi}\veccmanif$, the former in the
$\auhphi{\cdot}{\cdot}$ inner product and the latter in the Frobenius one:
\[
    \auhphi{\tG}{\tW} = \funch'(\tPhi)[\tW] = \innersmall{\gradFN\funch(\tPhi)}{\tW},
    \qquad \forall \tW \in \Tang_{\tPhi}\veccmanif .
\]
Since $\tG \in \Tang_{\tPhi}\veccmanif$, this is equivalent to the linear system on
$\Tang_{\tPhi}\veccmanif$
\begin{equation}
    \label{eq:change-of-metric-disc}
    \PAhPN\, \tG = \gradFN\funch(\tPhi),
    \qquad
    \PAhPN \coloneqq \Projvecn_{\tPhi}\, \matauhphi\, \Projvecn_{\tPhi},
\end{equation}
i.e., the discrete counterpart of \eqref{eq:change-of-metric-L2}. It is also the
exact analogue of the system with $\PAhPMrphi$ used in
\Cref{sec:multi-grad-Hu}, with the projection onto $\Tang_{\tPhi}\vecmanif$
replaced by the projection onto the smaller space $\Tang_{\tPhi}\veccmanif$: this is
what makes the mass constraint enter the linear system itself, rather than being
restored afterwards by the correction $\Hop_\phi(\phi\Sigma_{\phi,\phi})$ of
\eqref{eq:multi-grad-Hu}. The operator $\PAhPN$ is symmetric positive definite on
$\Tang_{\tPhi}\veccmanif$, so \eqref{eq:change-of-metric-disc} is solved by
preconditioned CG, with the preconditioner
$\Projvecn_{\tPhi}\,\preconapp\,\Projvecn_{\tPhi}$ obtained from the
preconditioner $\Ppreconapp$ of \Cref{sec:precond_tt} by the same replacement.
The cost is a single CG solve per gradient, independent of $p$.

\begin{remark}[Comparison of the two routes]
    \label{rem:L2-soundness}
    The two constructions differ in their robustness with respect to the
    discretization. Route~I is the discretization of a well-posed continuous
    formula: \eqref{eq:multi-grad-Hu} makes sense on $\cmanif$ for every $\phi$,
    because $a_\phi$ is coercive on $\Hspace$. Route~II is not: what
    \eqref{eq:change-of-metric-disc} computes is the \emph{discrete} $L^2$
    Riemannian gradient, which exists for every $\tPhi \in \veccmanif$ but whose
    continuous counterpart \eqref{eq:change-of-metric-L2} requires the additional
    assumption $\Aop_\phi\phi \in \Lspace$. One should therefore not expect
    route~II to be robust under mesh refinement in the sense of
    \cite{altmann_riemannian_2025}. We nonetheless use it in all experiments,
    since it is roughly $p+1$ times cheaper.
\end{remark}

\subsection{Implementation details}
\label{sec:multi-implementation}
The line-search methods of \Cref{sec:linesearch} are used unchanged; only the
retraction \eqref{eq:retraction} is modified, the normalization being performed
component-wise as in \eqref{eq:retraction-OB}, as detailed below.

Both routes require the component-wise inner products $\dinner{\cdot}{\cdot}$,
which at first sight are $p$ separate contractions of $d$-dimensional tensors. In
all the places where they occur in the algorithms above, however, one of the two
arguments is the current iterate $\phi$ itself, and they can therefore be
computed with almost no computational cost.

\begin{lemma}
    \label{lem:last-core}
    Let $\tPhi \in \vecmanif$ with cores $G_1, \dots, G_{d+1}$, and assume
    $\tPhi$ left-orthogonal in the sense of \eqref{eq:left-orth}, i.e., that
    $(\mtG_k^{\mathsf{L}})^\top\mtG_k^{\mathsf{L}} = \mtI_{r_k}$ for $k = 1,
        \dots, d$.
    Let $\tW \in \Tang_{\tPhi}\vecmanif$ be parametrized by the $\delta$-cores
    $(\delta W_1, \dots, \delta W_{d+1})$. Then
    \begin{equation}
        \label{eq:last-core}
        \dinner{\tW}{\tPhi}_{jj} = \delta W_{d+1}(j)^\top G_{d+1}(j),
        \qquad j = 1, \dots, p .
    \end{equation}
\end{lemma}
\begin{proof}
    By the proof of \Cref{prop:multi-intersection}, the gauged representation of
    $\tPhi$ as an element of $\Tang_{\tPhi}\vecmanif$ is given by the
    $\delta$-cores $(\delta G_1, \dots, \delta G_{d+1}) = (0, \dots, 0,
        G_{d+1})$. We write, as in \eqref{eq:TT-interface},
    \[
        \tW(i_1, \dots, i_d, j) = \sum_{k=1}^{d+1}\tPhi_{\leq k-1}(i_1, \dots, i_{k-1})\,
        \delta W_k(i_k)\, \tPhi_{\geq k+1}(i_{k+1}, \dots, i_d, j)^\top,
    \]
    where the interface matrices have been introduced in \eqref{eq:interface} and $i_{d+1} = j$. Contracting over $i_1, \dots, i_d$, the term $k = d+1$ gives
    \[
        \sum_{i_1, \dots, i_d} \delta W_{d+1}(j)^\top \tPhi_{\leq d}(\is)^\top \tPhi_{\leq d}(\is)\, G_{d+1}(j)
        = \delta W_{d+1}(j)^\top G_{d+1}(j),
    \]
    since $\tPhi_{\leq d}^\top \tPhi_{\leq d} = \mtI_{r_d}$ by
    left-orthogonality. For $k \leq d$, the factors $\tPhi_{\geq k+1}$ do not depend on
    $i_k$, so the sum over $i_k$ can be carried out first and produces $\sum_{i_k}
        \delta W_k(i_k)^\top G_k(i_k) = [\delta W_k, G_k] = 0$ by the gauge
    conditions. Hence all terms with $k \leq d$ vanish.

\end{proof}

The point of \Cref{lem:last-core} is that the tensor operations which are
needed for the multicomponent case, and which are not part of the standard TT
toolbox, are inexpensive: they all act on the last core alone. Indeed, by
\eqref{eq:last-core} the computational cost of $\dinner{\tW}{\tPhi}$ is that of a
single contraction of the last cores, i.e., $\calO(r_d\, p)$ operations, instead
of the $p$ contractions of $d$-dimensional tensors that its definition suggests.
Similarly, the rank-one updates $\phi\,\Sigma$ appearing in
\eqref{eq:multi-projN}, \eqref{eq:proj-L2} and \eqref{eq:multi-grad-Hu} amount to
rescaling the last core of $\tPhi$ column-wise by $\vsigma$, and the frames
$\hphi_j$ of \eqref{eq:frame} are obtained by zeroing all but the $j$-th column of
$G_{d+1}$. The same observation gives the retraction: as in \eqref{eq:scheme}, a
step is followed by TT rounding to rank $\ur$ and then by the component-wise
normalization \eqref{eq:retraction-OB}, which only rescales the columns of the
last core. None of these operations affects the cost per iteration, which
remains dominated by the applications of $\matauhphi$ within the conjugate
gradient iteration.

Two further remarks on the discretization.

\begin{remark}[TT rank in the component mode]
    \label{rem:rank-component}
    The rank $r_d$, which separates the spatial modes from the component mode, is
    bounded by the number of components: since $r_{d+1} = 1$ and the last mode has
    size $p$,
    \[
        r_d \leq r_{d+1}\, p = p .
    \]
    Equivalently, the $d$-th unfolding of $\tPhi$ has only $p$ columns. In our
    numerical experiments we always take $r_d = p$, its maximal value.
\end{remark}

\begin{remark}[Preconditioner]
    The preconditioner $\precon = \stiffh + \mtV_0$ of \eqref{eq:Btens} is
    extended to the multicomponent case as $\precon_p = \precon \otimes \mtI_p$,
    consistently with \eqref{eq:multi-operators}. The eigenvalue tensor of
    \Cref{sec:precond_tt} then becomes
    \[
        \tK(i_1, \dots, i_d, j) = (\vlambda_1)_{i_1} + \dots + (\vlambda_d)_{i_d} + 1,
    \]
    so that the exponential sum \eqref{eq:tKinvapp_def} is applied on the shifted
    range $[\min(\tK) + 1, \max(\tK) + 1]$ and acquires one extra rank-one core
    $e^{-\tilde\alpha_i \ve_p}$. Everything else in \Cref{sec:precond} carries
    over verbatim.
\end{remark}

\section{Numerical experiments}
\label{sec:numexp}
We assess the performance of the proposed method on the Gross--Pitaevskii (GP)
equation, both in the single-component and in the multicomponent case.

\subsection{Setup}
\label{sec:exp_setup}
The $a_u$ Sobolev Gradient Flow (SGF) on the manifold of unit-norm fixed TT-rank
functions $\cmanif = \manif \cap \sphere$ is discretized following
\Cref{sec:numerical_discretization}. We compare it with the $a_u$ and $\rmH^1$
SGF without tensor compression, referred to as ``full'', whenever the latter are
feasible. For the full $\rmH^1$ SGF, linear systems with $\stiffh$ are solved
directly by fast diagonalization, while in the full $a_u$ SGF, linear systems
with $\matauh$ are solved by preconditioned conjugate gradient. In both the TT
and the full case, ``None'' denotes no preconditioning, ``S'' the preconditioner
based on the stiffness matrix, and ``S+V'' the preconditioner incorporating also
the approximation of the potential, as introduced in
\Cref{sec:precond_full,sec:precond_tt}.

In all figures we report the relative energy error $\abs{E(u_k) -
        E^\star}/\abs{E^\star}$, where both the energy of the current iterate and the
reference value are re-evaluated with the higher-order quadrature rule of
\Cref{rmk:highorder_quadrature}, which integrates the nonlinearity exactly.
This re-evaluation serves plotting purposes only, as it displays the true energy
error rather than one polluted by the under-integration of the SEM quadrature;
the optimization itself, including the stopping criterion, is carried out
entirely with the standard SEM quadrature rule. The reference value $E^\star$ is
computed on a finer grid and with higher TT ranks $r$, after checking that it is
indeed the lowest energy attained over all the runs of the experiment at hand,
the full-format ones included.

\paragraph{Automatic differentiation of the TT cores.}
The most time-consuming operation is the matrix-vector multiplication
\[
    \PAhPMr \tZ = \Projvecm_{\tU}\,\matauh \,\Projvecm_{\tU} \tZ,
    \qquad \tZ\in \Tang_{\tU}\vecmanif.
\]
Its efficient implementation across the whole range of mode sizes and TT
ranks we consider is a nontrivial task, which is why
we avoid implementing it altogether.
The only two quantities we implement by hand are the scalar functions
\begin{equation}
    \label{eq:scalar-implemented}
    \funch(\tX)
    \qquad \text{and} \qquad
    f(\tX) = \inner{\matauh \tZ}{\tX},
\end{equation}
namely the discrete energy \eqref{eq:E_discrete} and the discrete bilinear form
\eqref{eq:au_discrete}, both evaluated on TT tensors. Their implementation is
straightforward with the library
\texttt{ttax}\footnote{\url{https://github.com/fasghq/ttax}}, a tensor-train
toolbox built on \texttt{jax} \cite{bradbury_jax_2018}, and consists of a few
lines of TT arithmetic. Everything else is obtained by automatic
differentiation, following \cite{novikov_automatic_2022}: differentiating
\eqref{eq:scalar-implemented} with respect to the TT cores of $\tX$, at $\tX =
    \tU$, returns the projected quantities $\PAhPMr \tZ$ and the Frobenius
Riemannian gradient $\gradF\funch(\tU)$ of \eqref{eq:grad-F}
directly in the $\delta$-parametrization of $\Tang_{\tU}\vecmanif$ used in
\eqref{eq:delta-parametrization}. This is efficient because reverse-mode
automatic differentiation costs a small constant times one evaluation of the
function itself, and because the $\delta G_k$ parametrization keeps all
operations within the tangent space, with no conversion to full tensors of
dimension $n^d$. One projected matrix-vector product therefore costs
asymptotically as much as one evaluation of $f$, and requires no hand-derived
formula. The evaluation is finally wrapped in the just-in-time compilation of
\texttt{jax}, which traces the computational graph once per combination of
dimension, mode sizes and TT ranks and leaves to the compiler the order of the
tensor contractions, whose optimal choice depends on the relative size of $n$
and $r$. Thanks to the \texttt{jax} backend, the same code runs unchanged on CPU
and GPU.

\paragraph{Software, parameters, and hardware.}
For Riemannian optimization we rely on the package Pymanopt
\cite{townsend_pymanopt_2016}. All runs start from the Thomas--Fermi
approximation and are stopped when the norm of the Riemannian gradient falls
below $10^{-6}$, or after $2000$ iterations. The linear systems of
\Cref{prop:projau_Hu} and \eqref{eq:change-of-metric-disc} are solved by
preconditioned CG with tolerance $10^{-10}$ and at most $200$ iterations, and
the exponential sum \eqref{eq:tKinvapp_def} uses $k=10$ terms in all
experiments. All experiments were run in double precision on a single NVIDIA RTX
4090 GPU with 24 GB of memory. The code reproducing all the experiments of this
section will be made publicly available upon publication at
\url{https://github.com/IvanBioli/TT_for_GPE.git}.

\subsection{Algorithmic choices}
\label{sec:exp_algo}

We first compare the optimizers of \Cref{sec:linesearch} and the preconditioners
of \Cref{sec:precond} on the three-dimensional GP equation with a harmonic
potential
\[
    V(x) = |x|^2,
\]
on the domain $\Omega = (-6, 6)^3$ with $\beta = 1000$,
following~\cite{heid_gradient_2021}. In both cases we use $n=400$ grid points
per dimension and polynomial degree $k=4$, so that the tensors involved have
$n^3$ entries.

\subsubsection{Comparison of Riemannian Gradient Descent and Riemannian Nonlinear Conjugate Gradient}
\label{sec:exp_optimizer}
\Cref{fig:exp1_CGvsSD} reports the energy error against iterations and against
computational time, for the full format and for the TT format with $r = 5, 10,
    15$. R-NLCG reduces the iteration count with respect to R-GD, by about $30\%$
in the full format and at rank $r=10$. Its per-iteration cost is, however,
slightly higher, because of the transport of the previous search direction and
of the additional inner products in \eqref{eq:betaHS}.
From the point of view of computational time, the performance of the
two methods is therefore similar.

The advantage of R-NLCG is more evident on the other problems. In the
multicomponent case of \Cref{sec:multi-experiments}, where the same comparison
is reported in \Cref{tab:exp9_multicomponent}, R-NLCG roughly halves the
iteration count and remains clearly faster also in time, by up to a factor of
$\approx 3$. We observed the same behaviour, although we do not report it, for the
potentials of \Cref{sec:exp_fullTT}: the conjugate directions help the most when
the gap between the first and the second eigenvalue is small, and the harmonic
potential considered here is the most benign example in this respect.

In view of the lower iteration count and of the discussion above,
we use R-NLCG in most of the experiments that follow.
\begin{figure}
    \begin{subfigure}[b]{0.49\textwidth}
        \includegraphics[width=\textwidth]{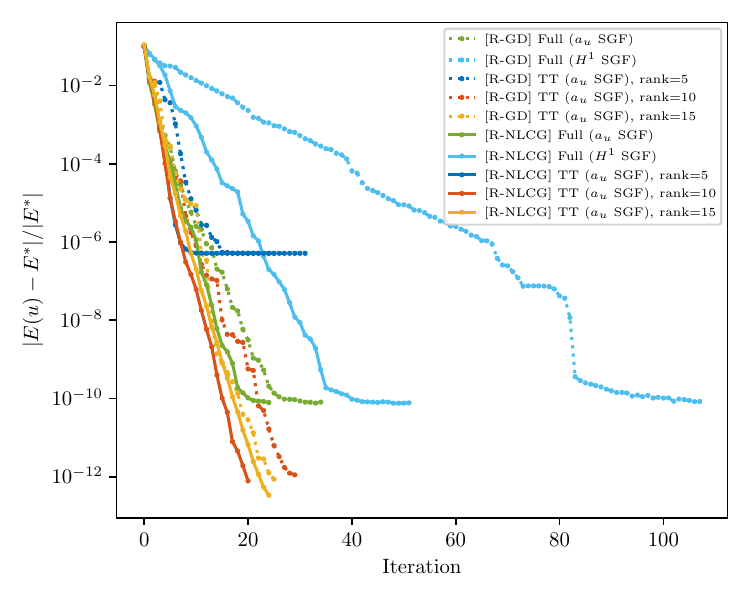}
        \caption{Energy error vs iterations.}
    \end{subfigure}
    \hfill
    \begin{subfigure}[b]{0.49\textwidth}
        \includegraphics[width=\textwidth]{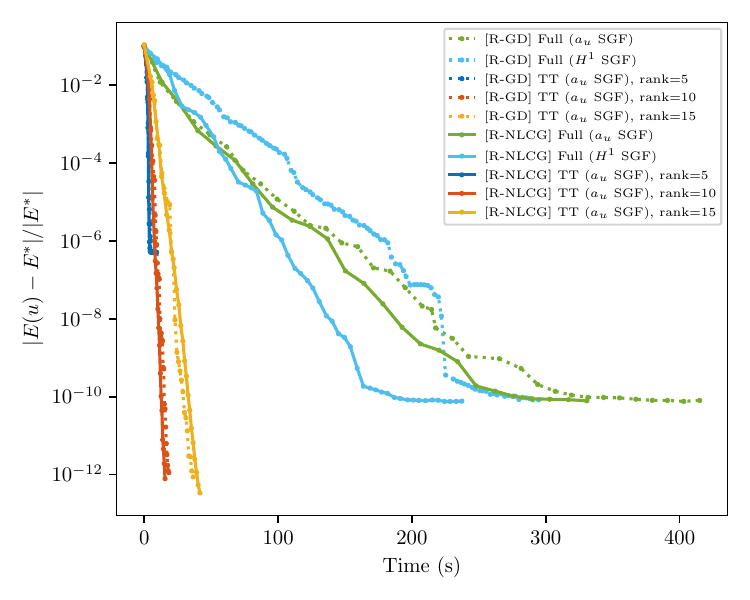}
        \caption{Energy error vs time.}
    \end{subfigure}
    \caption{Comparison of Riemannian Gradient Descent (R-GD) and Riemannian
        Nonlinear Conjugate Gradient (R-NLCG) for the $a_u$ SGF, in the full
        format and in the TT format; the full $\rmH^1$ SGF is also reported
        for reference. Three-dimensional GP equation with harmonic potential
        $V(x)=|x|^2$ on $\Omega=(-6,6)^3$ and $\beta=1000$, discretized with
        $n=400$ grid points per dimension and polynomial degree $k=4$. In the
        plot of the energy error against time, each dot corresponds to one
        iteration.}
    \label{fig:exp1_CGvsSD}
\end{figure}

\subsubsection{Comparison of preconditioners}
\label{sec:exp_precond}
\Cref{tab:exp2_precond} compares the preconditioners introduced in
\Cref{sec:precond}.
The number of iterations of Riemannian NLCG is approximately the same for
all preconditioners, as expected since the preconditioners only affect the
number of inner CG iterations that are needed to compute the gradient.
The number of those inner iterations is listed in the last column, and it can be
seen that the effect of preconditioning is more pronounced for the optimization
in the full tensor format, where using the ``S+V'' preconditioner reduced the
average number of CG iterations by an order of magnitude and the total time by a
factor of about $6$. In the TT format the reduction in CG iterations is
comparable, from about $150$ to about $20$, but the gain in time is smaller,
around a factor of $2$: applying the preconditioner in the split form
\eqref{eq:precon_split} is itself a sequence of $k=10$ projected products, whose
cost is of the same order as one For the same reason using the ``S''
preconditioner in the TT format can be slower than using no preconditioning at
all. For these reasons, we use the ``S+V'' preconditioner throughout the rest of
the numerical experiments.

\begin{table}[htbp]
    \caption{Comparison of the preconditioners of \Cref{sec:precond} for the
        $a_u$ Sobolev Gradient Flow, in the full format and in the TT format. Three-dimensional GP equation with harmonic potential
        $V(x)=|x|^2$ on $\Omega=(-6,6)^3$ and $\beta=1000$, discretized with
        $n=400$ grid points per dimension and polynomial degree $k=4$, R-NLCG
        optimizer. ``None'' denotes no preconditioning, ``S'' the preconditioner
        based on the stiffness matrix, and ``S+V'' the preconditioner
        incorporating also the approximation of the potential. The last column
        reports the average number of inner CG iterations needed to compute the
        Riemannian gradient.}
    \label{tab:exp2_precond}
    \centering
    \begin{tabular}{llcrrrr}
        \toprule
        Method                & Rank                & Precond & Final $E(u)$ & Time (s) & Iters & Average CG iters \\
        \midrule
        \multirow{3}{*}{Full} &                     & None    & 6.308835070  & 1970     & 26    & 191              \\
                              &                     & S       & 6.308835070  & 827      & 24    & 46.2             \\
                              &                     & S+V     & 6.308835070  & 331      & 24    & 14.2             \\
        \midrule
        \multirow{9}{*}{TT}   & \multirow{3}{*}{5}  & None    & 6.308838327  & 18.1     & 24    & 143              \\
                              &                     & S       & 6.308838327  & 15.1     & 23    & 51.5             \\
                              &                     & S+V     & 6.308838327  & 9.16     & 25    & 17.2             \\
        \cmidrule{2-7}
                              & \multirow{3}{*}{10} & None    & 6.308835069  & 27.2     & 20    & 162              \\
                              &                     & S       & 6.308835069  & 32.9     & 20    & 57.2             \\
                              &                     & S+V     & 6.308835069  & 15.5     & 20    & 18.9             \\
        \cmidrule{2-7}
                              & \multirow{3}{*}{15} & None    & 6.308835069  & 54.8     & 24    & 157              \\
                              &                     & S       & 6.308835069  & 84.6     & 24    & 61.3             \\
                              &                     & S+V     & 6.308835069  & 41.6     & 24    & 22.8             \\
        \bottomrule
    \end{tabular}
\end{table}

\subsection{Comparison of the full and TT formats}
\label{sec:exp_fullTT}

We now compare the full and TT formats on three potentials of increasing
``complexity''.
In all cases, we use the R-NLCG optimizer and the ``S+V'' preconditioner.

\subsubsection{Harmonic potential}
\label{sec:exp_harmonic}
We keep the setting of \Cref{sec:exp_algo} and let the mesh width vary. The
results for $n=400, 800$ and $k=4$ are shown in \Cref{fig:exp3_fullvsTT}. At $n=400$ the
TT format at rank $r=10$ reaches the same energy as the full format with a
$20$-fold reduction in computational time. This is the optimal rank among those
that we tested and in the range of energy errors that we observe.
Using tensors with fixed TT-rank $r=5$ gives an approximation that is accurate
up to $3\cdot 10^{-6}$ in relative error, at half the computational cost of
using rank $r=10$. At around $3\cdot 10^{-6}$, the error reaches a plateau.
Finally, computations with TT-rank $r=15$ do not yield significant further accuracy
gains in the range we tested, while increasing the computational time.

When we use $n=800$ points per direction in the discretization, the optimization
with full-size tensors exceeds the memory of our machine, and we are only able to
run computations in the TT format, at a cost that grows mildly with the
one-dimensional grid size: with rank $r=10$ the time doubles when $n$ doubles,
consistent with the linear-in-$n$ cost of the TT format. This is against the
cubic growth of the full format.

\begin{figure}
    \begin{subfigure}[b]{0.49\textwidth}
        \includegraphics[width=\textwidth]{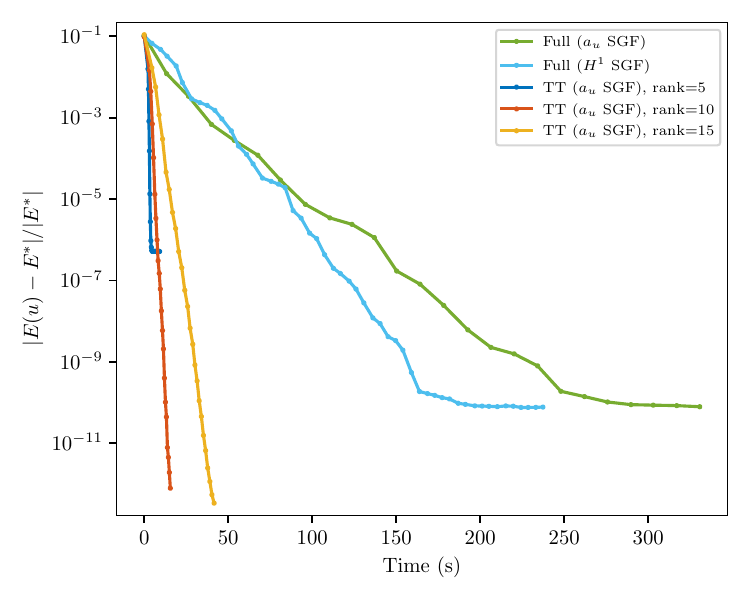}
        \caption{$n=400$ points per dimension.}
    \end{subfigure}
    \hfill
    \begin{subfigure}[b]{0.49\textwidth}
        \includegraphics[width=\textwidth]{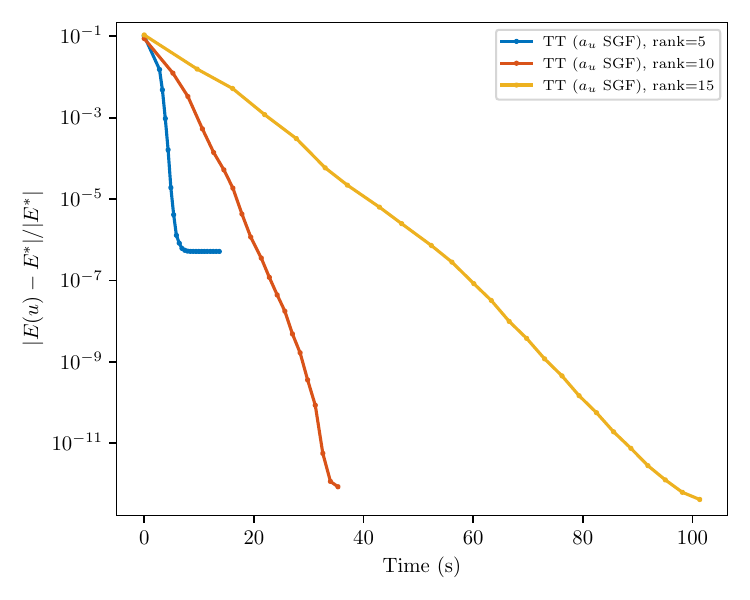}
        \caption{$n=800$ points per dimension.}
    \end{subfigure}
    \caption{Comparison of the full and TT formats for the $a_u$ SGF, with TT
        ranks $r=5,10,15$; the full $\rmH^1$ SGF is reported for reference.
        Three-dimensional GP equation with harmonic potential $V(x)=|x|^2$ on
        $\Omega=(-6,6)^3$ and $\beta=1000$, discretized with polynomial degree
        $k=4$, R-NLCG optimizer and S+V preconditioner. Energy error vs time is
        shown for $n=400$ and $n=800$ grid points per dimension; each dot
        corresponds to one iteration. At $n=800$ optimization exceeds the memory
        of the machine, and only the TT format is reported.}
    \label{fig:exp3_fullvsTT}
\end{figure}

\subsubsection{Lattice potential}
\label{sec:exp_lattice}
Next, we consider the GP equation with a lattice potential similar to the one
used in~\cite{heid_gradient_2021}, but with a higher frequency and in three
dimensions:
\[
    V(x) = |x|^2  + 40 + 40 \sin (10\pi x_1) \sin (10\pi x_2) \sin (10\pi x_3),
\]
again with $\Omega = (-6, 6)^3$ and $\beta = 1000$. Results for $n=400, 800$ and
$k=4$ are shown in \Cref{fig:exp4_lattice}. The oscillations in the potential
make the problem harder in two respects: the number of iterations to convergence
roughly doubles with respect to the harmonic case, and the final energy error is
higher. The comparison between the two formats is nevertheless unchanged: at
$n=400$ the TT format at rank $r=10$ matches the accuracy of the full format
while being about $10$ times faster, and at $n=800$ only the TT format is feasible on our
machine. However, in this case increasing the rank to $r=15$ leads to an
improvement in accuracy at $n=800$. This is consistent with the more complex
nature of the potential (and hence of the minimizer of the GP energy).

\begin{figure}
    \begin{subfigure}[b]{0.49\textwidth}
        \includegraphics[width=\textwidth]{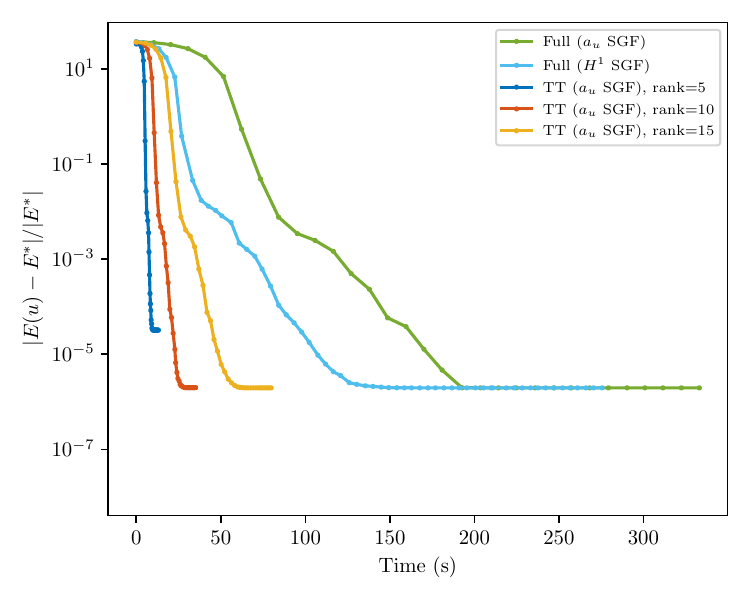}
        \caption{$n=400$ points per dimension.}
    \end{subfigure}
    \hfill
    \begin{subfigure}[b]{0.49\textwidth}
        \includegraphics[width=\textwidth]{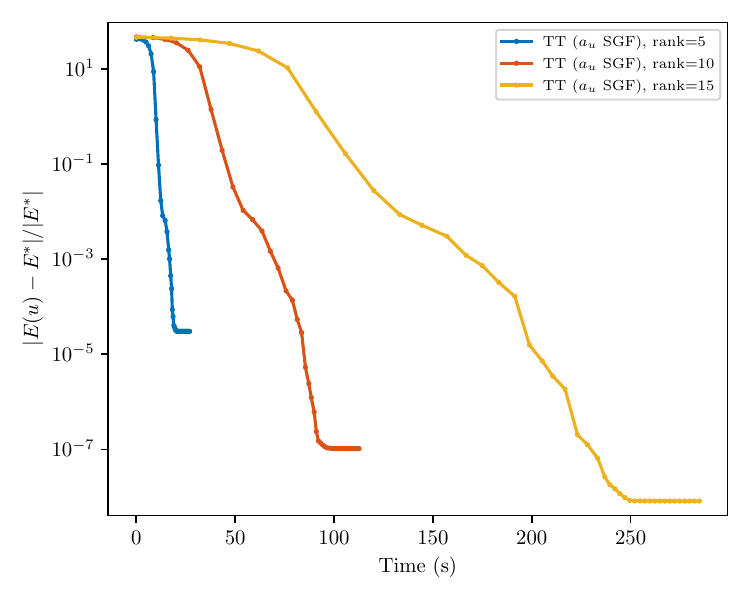}
        \caption{$n=800$ points per dimension.}
    \end{subfigure}
    \caption{Comparison of the full and TT formats for the $a_u$ SGF, with TT
        ranks $r=5,10,15$; the full $\rmH^1$ SGF is reported for reference.
        Three-dimensional GP equation with lattice potential $V(x) = |x|^2 + 40 +
            40 \sin(10\pi x_1)\sin(10\pi x_2)\sin(10\pi x_3)$ on $\Omega=(-6,6)^3$
        and $\beta=1000$, discretized with polynomial degree $k=4$, R-NLCG
        optimizer and S+V preconditioner. Energy error vs time is shown for
        $n=400$ and $n=800$ grid points per dimension; each dot corresponds to one
        iteration. At $n=800$ optimization exceeds the memory of the machine, and
        only the TT format is reported.}
    \label{fig:exp4_lattice}
\end{figure}

\subsubsection{Anderson localization in two dimensions}
\label{sec:exp_anderson}
Finally, we consider the Anderson localization effect, in the setting
of~\cite{henning_sobolev_2020} with $\Omega = (-6,6)^2$, $\beta = 10$. The domain
$\Omega$ is split into $40^2$ squares of equal size, and in each of the squares
the value of $V$ is chosen with equal probability to be either~$1$
or~$\varepsilon^{-2}$, for $\varepsilon=0.03$. The grid has to resolve each of
the squares, so that the coarsest grid we choose contains $n=400$ points per
dimension, and the potential tensor is a sum of $R = 40$ rank-$1$ terms in TT
format. This is the worst case among our examples: by the estimates of
\Cref{sec:numerical_discretization}, the potential enters the cost through
the Hadamard product $\tV \odot \tW$, which contributes $\calO(dnR^2r^2)$ and
increases the TT-rank by a factor $R$.

Results for $n=1600, 3200, 6400$ and $k=4$ are shown in
\Cref{fig:exp5_anderson}. Since here $d=2$, the full format remains feasible on
all these grids, and the comparison isolates the effect of the different scaling
in $n$: the two formats are on par at $n=1600$, whereas at $n=6400$ the TT
format is substantially faster. The ranks needed are much larger than in
the previous examples: rank $r=20$ is not sufficient and ranks $r$ between $40$
and $60$ are required to match the energy of the full format, which is
consistent with the TT rank $R=40$ of the potential. Finally, the $\rmH^1$ SGF,
which does not use the energy-adaptive metric, needs more than a thousand
iterations against the twenty-odd of the $a_u$ SGF on every grid, and is the
slowest method throughout.

\begin{figure}[htbp]
    \begin{subfigure}[b]{0.32\textwidth}
        \includegraphics[width=\textwidth]{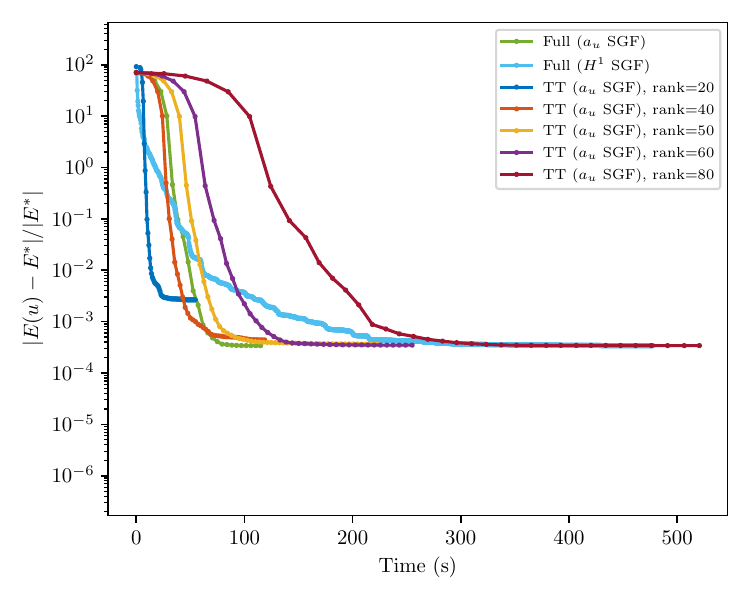}
        \caption{$n=1600$ points per dimension.}
    \end{subfigure}
    \hfill
    \begin{subfigure}[b]{0.32\textwidth}
        \includegraphics[width=\textwidth]{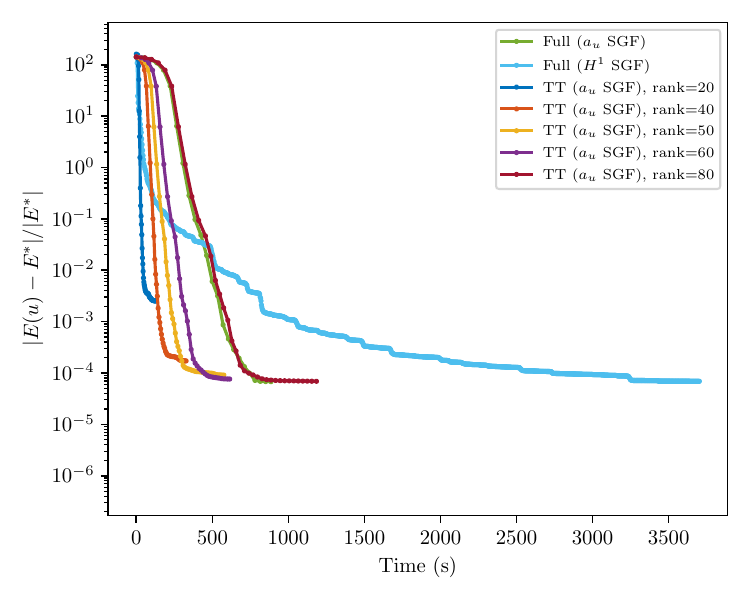}
        \caption{$n=3200$ points per dimension.}
    \end{subfigure}
    \hfill
    \begin{subfigure}[b]{0.32\textwidth}
        \includegraphics[width=\textwidth]{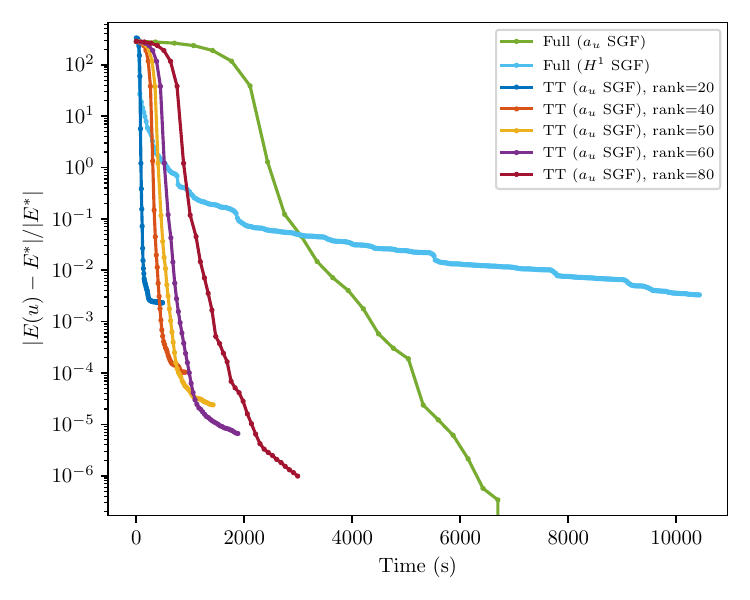}
        \caption{$n=6400$ points per dimension.}
    \end{subfigure}
    \caption{Comparison of the full and TT formats for the $a_u$ SGF, with TT
        ranks $r=20,40,50,60,80$; the full $\rmH^1$ SGF is reported for reference.
        Two-dimensional GP equation on $\Omega=(-6,6)^2$ with $\beta=10$,
        with the Anderson localization potential
        taking
        the values $1$ or $\varepsilon^{-2}$ with equal probability on each of
        $40^2$ equal squares, for $\varepsilon=0.03$; discretization with
        polynomial degree $k=4$, R-NLCG optimizer and S+V preconditioner. Energy
        error vs time is shown for $n=1600$, $n=3200$, and $n=6400$ grid points per
        dimension; each dot corresponds to one iteration.}
    \label{fig:exp5_anderson}
\end{figure}

\subsection{Multicomponent Gross--Pitaevskii equation}
\label{sec:multi-experiments}
As a final numerical example, we consider a three-component Bose--Einstein
condensate (BEC) model on the unit square $\Omega = [0,1]^2$ with a periodic
potential, following a setup similar to the one in \cite{altmann_riemannian_2025}. The
interaction parameters are $\kappa_{11}=0.5$, $\kappa_{22}=5$,
$\kappa_{33}=10$, and $\kappa_{ij}=1$ otherwise. The mass of each component
(i.e., the number of particles in the BEC model) is $N_1=N_2=N_3=1$, and $\beta
    = 10$. The potential assumes values in $\{0, 2^{12}\}$ on a checkerboard pattern
with square edge length $\epsilon=2^{-6}$.
To enforce small values near the
boundary, we additionally include the trapping potential
$V_{\mathrm{trap}}(x_1,x_2) = 10^{6}\left((2x_1-1)^{40} + (2x_2-1)^{40}\right)$.
The Riemannian gradient is computed by route~II of
\Cref{sec:multi-grad-PAPN}, and the rank in the component mode is $r_d = p = 3$,
as discussed in \Cref{rem:rank-component}.

Results for $n=2^{10}, 2^{11}$ and $k=4$ are shown in
\Cref{fig:exp9_multicomponent}. The experimental findings of the
single-component case are confirmed also in this case: at $n=2^{11}$ the TT
format at rank $r=10$ attains the energy of the full format while being $9$
times faster. Using tensors with fixed TT-rank $r=5$ gives an energy that is
accurate up to $4\cdot 10^{-6}$ in relative error, while being more than $20$
times faster than the optimization with full-size tensors. In
\Cref{tab:exp9_multicomponent} we report the comparison of the two optimization
methods presented in this paper on the same problem: R-NLCG roughly halves both
the number of iterations and the computational time with respect to R-GD, for
both formats and for all ranks, showing the effectiveness of the conjugate
directions in this more challenging problem.

\begin{figure}[htbp]
    \begin{subfigure}[b]{0.49\textwidth}
        \includegraphics[width=\textwidth]{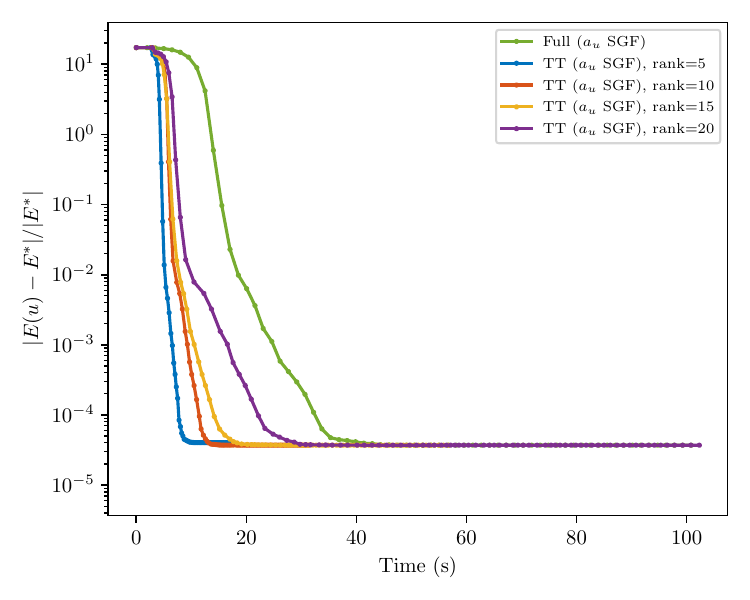}
        \caption{$n=2^{10}$ points per dimension.}
    \end{subfigure}
    \hfill
    \begin{subfigure}[b]{0.49\textwidth}
        \includegraphics[width=\textwidth]{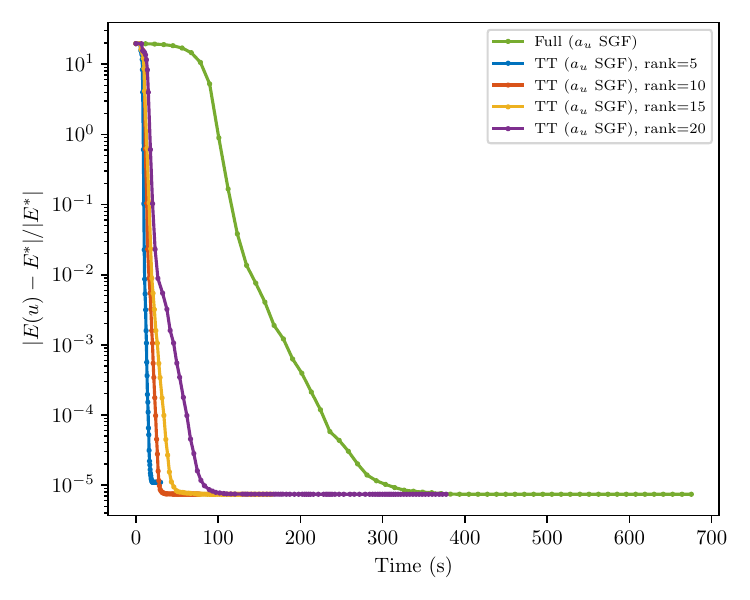}
        \caption{$n=2^{11}$ points per dimension.}
    \end{subfigure}
    \caption{Comparison of the full and TT formats for the $a_u$ SGF, with TT
    ranks $r=5,10,15,20$. Three-component Bose--Einstein condensate on
    $\Omega=[0,1]^2$ with a checkerboard potential taking the values $\{0,
        2^{12}\}$ on squares of edge length $\epsilon=2^{-6}$ plus the trapping
    potential $V_{\mathrm{trap}}(x_1,x_2) = 10^{6}((2x_1-1)^{40} +
        (2x_2-1)^{40})$, interaction parameters $\kappa_{11}=0.5$,
    $\kappa_{22}=5$, $\kappa_{33}=10$ and $\kappa_{ij}=1$ otherwise, masses
    $N_1=N_2=N_3=1$ and $\beta=10$; discretization with polynomial degree
    $k=4$, R-NLCG optimizer, S+V preconditioner and rank $r_d = p = 3$ in the
    component mode. Energy error vs time is shown for $n=2^{10}$ and $n=2^{11}$
    grid points per dimension; each dot corresponds to one iteration.}
    \label{fig:exp9_multicomponent}
\end{figure}

\begin{table}[htbp]
    \caption{Comparison of Riemannian Gradient Descent (R-GD) and Riemannian
    Nonlinear Conjugate Gradient (R-NLCG) for the $a_u$ Sobolev Gradient Flow,
    in the full format and in the TT format.
    Three-component Bose--Einstein condensate on $\Omega=[0,1]^2$ with the
    checkerboard and trapping potentials of \Cref{fig:exp9_multicomponent},
    interaction parameters $\kappa_{11}=0.5$, $\kappa_{22}=5$,
    $\kappa_{33}=10$ and $\kappa_{ij}=1$ otherwise, masses $N_1=N_2=N_3=1$ and
    $\beta=10$; discretized with $n=2^{11}$ grid points per dimension and
    polynomial degree $k=4$, S+V preconditioner.}
    \label{tab:exp9_multicomponent}
    \centering
    \begin{tabular}{llcrrr}
        \toprule
        Optimizer               & Method              & Rank & Final $E(u)$ & Iters & Time (s) \\
        \midrule
        \multirow{5}{*}{R-GD}   & Full                &      & 4777.17      & 148   & 1089.68  \\
                                & \multirow{4}{*}{TT} & 5    & 4777.19      & 99    & 31.7662  \\
                                &                     & 10   & 4777.17      & 146   & 108.703  \\
                                &                     & 15   & 4777.17      & 201   & 445.981  \\
                                &                     & 20   & 4777.17      & 176   & 610.613  \\
        \midrule
        \multirow{5}{*}{R-NLCG} & Full                &      & 4777.17      & 60    & 675.394  \\
                                & \multirow{4}{*}{TT} & 5    & 4777.19      & 56    & 29.8981  \\
                                &                     & 10   & 4777.17      & 78    & 77.4799  \\
                                &                     & 15   & 4777.17      & 79    & 165.219  \\
                                &                     & 20   & 4777.17      & 98    & 377.234  \\
        \bottomrule
    \end{tabular}
\end{table}
\FloatBarrier
\section{Conclusions and perspectives}
We have shown how to construct a first-order Riemannian optimization scheme for
the computation of Gross--Pitaevskii ground states directly in the tensor train
format. The scheme operates on the manifold of unit-norm functions of fixed
TT-rank, so that both the mass constraint and the low-rank representation are
preserved along the whole optimization. Its descent direction is derived from
the energy-adaptive Riemannian gradient, whose computation requires the solution
of a linear system that we precondition. Combined with a spectral discretization
in space with numerical integration, this scheme substantially reduces the
computational time with respect to the full-rank computations. The reduction is
consistent across the discretizations and the experiments we considered, and
holds for both the single- and the multicomponent Gross--Pitaevskii equation.
The results also show that small to moderate TT-ranks are sufficient to capture
the ground state up to the relative error introduced by the discretization.

Selecting the optimal rank for a given target error is, however, not a trivial
task, and an a posteriori or rank-adaptive procedure for this choice would make
the method even more attractive. A further direction of improvement would be to
study the quantized scheme in the framework of Riemannian optimization: even a
mild reduction of the mode size and increase of the dimension (reshaping the
tensors to be, e.g., six- or twelve-dimensional, with correspondingly smaller
mode sizes) could prove beneficial without incurring the ill-conditioning
associated with QTT \cite{Bachmayr2020}. Finally, the scheme we propose is a
first-order one, and Riemannian second-order methods for Gross--Pitaevskii on
low-rank manifolds are a promising direction of future research; there, however,
Hessians are known to easily become indefinite
\cite[Prop.~2.28--2.30]{Bioli2024}, which is what prevented us from pursuing it
in the present work.

\bibliographystyle{abbrv}
\bibliography{bibliography}

\begin{thebibliography}{10}

\bibitem{Abraham1988}
R.~Abraham, J.~E. Marsden, and T.~Ratiu.
\newblock {\em Manifolds, Tensor Analysis, and Applications}, volume~75 of {\em
  Applied Mathematical Sciences}.
\newblock Springer-Verlag, New York, 2 edition, 1988.

\bibitem{absil_optimization_2008}
P.-A. Absil, R.~Mahony, and R.~Sepulchre.
\newblock {\em Optimization {{Algorithms}} on {{Matrix Manifolds}}:}.
\newblock Princeton University Press, Dec. 2008.

\bibitem{AiHenningYadavYuan2026}
Y.~Ai, P.~Henning, M.~Yadav, and S.~Yuan.
\newblock {Riemannian} conjugate {Sobolev} gradients and their application to
  compute ground states of {BECs}.
\newblock {\em Journal of Computational and Applied Mathematics}, 473:116866,
  2026.

\bibitem{AltmannJmethod}
R.~Altmann, P.~Henning, and D.~Peterseim.
\newblock The {$J$}-method for the {G}ross-{P}itaevskii eigenvalue problem.
\newblock {\em Numer. Math.}, 148(3):575--610, 2021.

\bibitem{altmann_riemannian_2025}
R.~Altmann, M.~Hermann, D.~Peterseim, and T.~Stykel.
\newblock {Riemannian} optimization methods for ground states of multicomponent
  {Bose--Einstein} condensates.
\newblock {\em IMA Journal of Numerical Analysis}, 46(4):1994--2033, 2026.

\bibitem{altmann_energy-adaptive_2022}
R.~Altmann, D.~Peterseim, and T.~Stykel.
\newblock Energy-adaptive {{Riemannian}} optimization on the {{Stiefel}}
  manifold.
\newblock {\em ESAIM: Mathematical Modelling and Numerical Analysis},
  56(5):1629--1653, 2022.

\bibitem{AltmannPeterseimStykel2024}
R.~Altmann, D.~Peterseim, and T.~Stykel.
\newblock {Riemannian} {Newton} methods for energy minimization problems of
  {Kohn--Sham} type.
\newblock {\em Journal of Scientific Computing}, 101(1), 2024.
\newblock Article 6.

\bibitem{antoine_efficient_2017}
X.~Antoine, A.~Levitt, and Q.~Tang.
\newblock Efficient spectral computation of the stationary states of rotating
  {{Bose-Einstein}} condensates by preconditioned nonlinear conjugate gradient
  methods.
\newblock {\em Journal of Computational Physics}, 343:92--109, 2017.

\bibitem{Bachmayr2023Low}
M.~Bachmayr.
\newblock Low-rank tensor methods for partial differential equations.
\newblock {\em Acta Numerica}, 32:1--121, 2023.

\bibitem{Bachmayr2020}
M.~Bachmayr and V.~Kazeev.
\newblock Stability of low-rank tensor representations and structured
  multilevel preconditioning for elliptic {PDE}s.
\newblock {\em Found. Comput. Math.}, 20(5):1175--1236, 2020.

\bibitem{bao2004computing}
W.~Bao and Q.~Du.
\newblock Computing the ground state solution of {{Bose}}--{{Einstein}}
  condensates by a normalized gradient flow.
\newblock {\em SIAM Journal on Scientific Computing}, 25(5):1674--1697, 2004.

\bibitem{Bioli2024}
I.~Bioli.
\newblock Preconditioned low-rank {R}iemannian optimization for multiterm
  linear matrix equations.
\newblock Master's thesis, EPFL, 2024.

\bibitem{bioli_preconditioned_2025}
I.~Bioli, D.~Kressner, and L.~Robol.
\newblock Preconditioned low-rank riemannian optimization for symmetric
  positive definite linear matrix equations.
\newblock {\em SIAM Journal on Scientific Computing}, 47(2):A1091--A1116, 2025.

\bibitem{BouComas2025}
A.~{Bou-Comas}, M.~{P{\l}odzie{\'n}}, L.~{Tagliacozzo}, and J.~{Jos{\'e}
  Garc{\'\i}a-Ripoll}.
\newblock {Quantics Tensor Train for solving Gross-Pitaevskii equation}.
\newblock {\em arXiv e-prints}, page arXiv:2507.03134, July 2025.
\newblock arXiv:2507.03134.

\bibitem{boumal_introduction_2023}
N.~Boumal.
\newblock {\em An {{Introduction}} to {{Optimization}} on {{Smooth
  Manifolds}}}.
\newblock Cambridge University Press, 1 edition, Mar. 2023.

\bibitem{bradbury_jax_2018}
J.~Bradbury, R.~Frostig, P.~Hawkins, M.~J. Johnson, C.~Leary, D.~Maclaurin,
  G.~Necula, A.~Paszke, J.~VanderPlas, S.~Wanderman-Milne, and Q.~Zhang.
\newblock {{JAX}}: Composable transformations of {{Python}}+{{NumPy}} programs,
  2018.

\bibitem{braess_approximation_2005}
D.~Braess and W.~Hackbusch.
\newblock Approximation of {$1/x$} by exponential sums in {$[1,\infty)$}.
\newblock {\em IMA Journal of Numerical Analysis}, 25(4):685--697, Oct. 2005.

\bibitem{Cances2010}
E.~Canc\`es, R.~Chakir, and Y.~Maday.
\newblock Numerical analysis of nonlinear eigenvalue problems.
\newblock {\em Journal of Scientific Computing}, 45(1--3):90--117, 2010.

\bibitem{CHQZ2}
C.~Canuto, M.~Y. Hussaini, A.~Quarteroni, and T.~A. Zang.
\newblock {\em Spectral methods}.
\newblock Scientific Computation. Springer, Berlin, 2007.
\newblock Evolution to complex geometries and applications to fluid dynamics.

\bibitem{chenFullyDiscretizedSobolev2024}
Z.~Chen, J.~Lu, Y.~Lu, and X.~Zhang.
\newblock Fully discretized {{Sobolev}} gradient flow for the
  {{Gross-Pitaevskii}} eigenvalue problem.
\newblock {\em Mathematics of Computation}, Nov. 2024.

\bibitem{chenConvergenceSobolevGradient2024}
Z.~Chen, J.~Lu, Y.~Lu, and X.~Zhang.
\newblock On the {{Convergence}} of {{Sobolev Gradient Flow}} for the
  {{Gross}}--{{Pitaevskii Eigenvalue Problem}}.
\newblock {\em SIAM Journal on Numerical Analysis}, 62(2):667--691, Apr. 2024.

\bibitem{Connor2026}
R.~J.~J. Connor, C.~W. Duncan, and A.~J. Daley.
\newblock Tensor network methods for the {Gross--Pitaevskii} equation on fine
  grids.
\newblock {\em New Journal of Physics}, 28(2):023203, 2026.

\bibitem{DanailaKazemi2010}
I.~Danaila and P.~Kazemi.
\newblock A new {Sobolev} gradient method for direct minimization of the
  {Gross--Pitaevskii} energy with rotation.
\newblock {\em SIAM Journal on Scientific Computing}, 32(5):2447--2467, 2010.

\bibitem{danaila_computation_2017}
I.~Danaila and B.~Protas.
\newblock Computation of {{Ground States}} of the {{Gross--Pitaevskii
  Functional}} via {{Riemannian Optimization}}.
\newblock {\em SIAM Journal on Scientific Computing}, 39(6):B1102--B1129, Jan.
  2017.

\bibitem{deville_high-order_2002}
M.~O. Deville, P.~F. Fischer, and E.~H. Mund.
\newblock {\em High-{{Order Methods}} for {{Incompressible Fluid Flow}}}.
\newblock Cambridge {{Monographs}} on {{Applied}} and {{Computational
  Mathematics}}. Cambridge University Press, Cambridge, 2002.

\bibitem{GarciaRipollPerezGarcia2001}
J.~J. Garc{\'i}a-Ripoll and V.~M. P{\'e}rez-Garc{\'i}a.
\newblock Optimizing {Schr{\"o}dinger} functionals using {Sobolev} gradients:
  Applications to quantum mechanics and nonlinear optics.
\newblock {\em SIAM Journal on Scientific Computing}, 23(4):1316--1334, 2001.

\bibitem{hackbusch_computation_2019}
W.~Hackbusch.
\newblock Computation of best {$L^{\infty}$} exponential sums for {$1/x$} by
  {{Remez}}' algorithm.
\newblock {\em Computing and Visualization in Science}, 20(1):1--11, Feb. 2019.

\bibitem{HaegemanEtAl2011}
J.~Haegeman, J.~I. Cirac, T.~J. Osborne, I.~Pi{\v z}orn, H.~Verschelde, and
  F.~Verstraete.
\newblock Time-dependent variational principle for quantum lattices.
\newblock {\em Physical Review Letters}, 107(7):070601, 2011.

\bibitem{Haegeman2016Unifying}
J.~Haegeman, C.~Lubich, I.~Oseledets, B.~Vandereycken, and F.~Verstraete.
\newblock Unifying time evolution and optimization with matrix product states.
\newblock {\em Physical Review B}, 94(16):165116, 2016.

\bibitem{heid_gradient_2021}
P.~Heid, B.~Stamm, and T.~P. Wihler.
\newblock Gradient flow finite element discretizations with energy-based
  adaptivity for the {{Gross-Pitaevskii}} equation.
\newblock {\em Journal of Computational Physics}, 436:110165, July 2021.

\bibitem{henningGrossPitaevskiiEquation2025}
P.~Henning and E.~Jarlebring.
\newblock The {{Gross}}--{{Pitaevskii Equation}} and {{Eigenvector
  Nonlinearities}}: {{Numerical Methods}} and {{Algorithms}}.
\newblock {\em SIAM Review}, 67(2):256--317, May 2025.

\bibitem{henning_sobolev_2020}
P.~Henning and D.~Peterseim.
\newblock Sobolev {{Gradient Flow}} for the {{Gross}}--{{Pitaevskii Eigenvalue
  Problem}}: {{Global Convergence}} and {{Computational Efficiency}}.
\newblock {\em Siam Journal On Numerical Analysis}, 58(3):1744--1772, Jan.
  2020.

\bibitem{holtz_manifolds_2012}
S.~Holtz, T.~Rohwedder, and R.~Schneider.
\newblock On manifolds of tensors of fixed {{TT-rank}}.
\newblock {\em Numerische Mathematik}, 120(4):701--731, Apr. 2012.

\bibitem{karniadakis_spectralhp_2005}
G.~Karniadakis and S.~Sherwin.
\newblock {\em Spectral/Hp {{Element Methods}} for {{Computational Fluid
  Dynamics}}}.
\newblock Oxford University Press, June 2005.

\bibitem{KochLubich2007}
O.~Koch and C.~Lubich.
\newblock Dynamical low-rank approximation.
\newblock {\em SIAM Journal on Matrix Analysis and Applications},
  29(2):434--454, 2007.

\bibitem{kressner_preconditioned_2016}
D.~Kressner, M.~Steinlechner, and B.~Vandereycken.
\newblock Preconditioned {{Low-rank Riemannian Optimization}} for {{Linear
  Systems}} with {{Tensor Product Structure}}.
\newblock {\em SIAM Journal on Scientific Computing}, 38(4):A2018--A2044, Jan.
  2016.

\bibitem{kressner_krylov_2010}
D.~Kressner and C.~Tobler.
\newblock Krylov {{Subspace Methods}} for {{Linear Systems}} with {{Tensor
  Product Structure}}.
\newblock {\em SIAM Journal on Matrix Analysis and Applications},
  31(4):1688--1714, Jan. 2010.

\bibitem{lancaster_variational_1991}
P.~Lancaster and Q.~Ye.
\newblock Variational and numerical methods for symmetric matrix pencils.
\newblock {\em Bulletin of the Australian Mathematical Society}, 43(1):1--17,
  Feb. 1991.

\bibitem{liu_simple_2024}
X.~Liu, J.~Shen, and X.~Zhang.
\newblock A simple {{GPU}} implementation of spectral-element methods for
  solving {{3D Poisson}} type equations on rectangular domains and its
  applications, June 2024.
\newblock arXiv:2310.00226.

\bibitem{lubichTimeIntegrationTensor2015}
C.~Lubich, I.~V. Oseledets, and B.~Vandereycken.
\newblock Time {{Integration}} of {{Tensor Trains}}.
\newblock {\em SIAM Journal on Numerical Analysis}, 53(2):917--941, Jan. 2015.

\bibitem{lynch_direct_1964}
R.~E. Lynch, J.~R. Rice, and D.~H. Thomas.
\newblock Direct solution of partial difference equations by tensor product
  methods.
\newblock {\em Numerische Mathematik}, 6(1):185--199, Dec. 1964.

\bibitem{maday_spectral_1989}
Y.~Maday and A.~T. Patera.
\newblock Spectral element methods for the incompressible {{Navier-Stokes}}
  equations.
\newblock In {\em {{IN}}: {{State-of-the-art}} Surveys on Computational
  Mechanics ({{A90-47176}} 21-64). {{New York}}}, pages 71--143, Jan. 1989.

\bibitem{MRS2022}
C.~Marcati, M.~Rakhuba, and C.~Schwab.
\newblock Tensor rank bounds for point singularities in {{R}}{$^{3}$}.
\newblock {\em Advances in Computational Mathematics}, 48(3), 2022.

\bibitem{montardini_low-rank_2023}
M.~Montardini, G.~Sangalli, and M.~Tani.
\newblock A low-rank isogeometric solver based on {{Tucker}} tensors.
\newblock {\em Computer Methods in Applied Mechanics and Engineering},
  417:116472, Dec. 2023.

\bibitem{Niedermeier2026}
M.~Niedermeier, A.~Moulinas, T.~Louvet, J.~L. Lado, and X.~Waintal.
\newblock Solving the gross-pitaevskii equation on multiple different scales
  using the quantics tensor train representation.
\newblock {\em Phys. Rev. Res.}, 8:023006, Apr. 2026.

\bibitem{novikov_automatic_2022}
A.~Novikov, M.~Rakhuba, and I.~Oseledets.
\newblock Automatic differentiation for {{Riemannian}} optimization on low-rank
  matrix and tensor-train manifolds.
\newblock {\em SIAM Journal on Scientific Computing}, 44(2):A843--A869, 2022.

\bibitem{oseledets2011tensor}
I.~V. Oseledets.
\newblock Tensor-train decomposition.
\newblock {\em SIAM Journal on Scientific Computing}, 33(5):2295--2317, 2011.

\bibitem{rakhuba2019low}
M.~Rakhuba, A.~Novikov, and I.~Oseledets.
\newblock Low-rank riemannian eigensolver for high-dimensional hamiltonians.
\newblock {\em Journal of Computational Physics}, 396:718--737, 2019.

\bibitem{sato_riemannian_2021}
H.~Sato.
\newblock {\em Riemannian {{Optimization}} and {{Its Applications}}}.
\newblock {{SpringerBriefs}} in {{Electrical}} and {{Computer Engineering}}.
  Springer International Publishing, Cham, 2021.

\bibitem{steinlechner_riemannian_2016}
M.~M. Steinlechner.
\newblock {\em Riemannian {{Optimization}} for {{Solving High-Dimensional
  Problems}} with {{Low-Rank Tensor Structure}}}.
\newblock PhD thesis, EPFL, Lausanne, 2016.

\bibitem{townsend_pymanopt_2016}
J.~Townsend, N.~Koep, and S.~Weichwald.
\newblock Pymanopt: A python toolbox for optimization on manifolds using
  automatic differentiation.
\newblock {\em Journal of Machine Learning Research}, 17(137):1--5, 2016.

\end{thebibliography}
\end{document}